\documentclass{amsart}

\usepackage{amssymb,amsmath,amsthm,latexsym,breqn,enumerate}

  \theoremstyle{definition}
  \newtheorem{definition}{Definition}[section]
  
  \newtheorem{example}[definition]{Example}
  \newtheorem{remark}[definition]{Remark}

  \theoremstyle{plain}
  \newtheorem{lemma}[definition]{Lemma}
  \newtheorem{proposition}[definition]{Proposition}
  \newtheorem{theorem}[definition]{Theorem}
  \newtheorem{corollary}[definition]{Corollary}
  \newtheorem*{conjecture}{Conjecture}

  \newcommand{\J}{{\mathfrak{J}}}

\begin{document}

\title[Anti-Jordan triple system]
{The universal Associative envelope of the anti-Jordan triple system of $n \times n$ matrices}

\author{Hader A. Elgendy}

\address{Department of Mathematics and Statistics, University of Saskatchewan, Canada}

\email{hae431@mail.usask.ca}

\date{\textit{\today}}

\begin{abstract}
We show that the universal associative enveloping algebra of the simple anti-Jordan triple system
of all $n \times n$ matrices $(n \ge 2)$ over an algebraically closed field of characteristic 0 is finite dimensional.
We investigate the structure of the universal envelope and focus on the monomial basis,  the structure constants, and the center.
We explicitly determine the decomposition of the universal envelope into matrix algebras.
We classify all finite dimensional irreducible representations of the simple anti-Jordan triple system, and show that the universal envelope is semisimple.
We also provide an example to show that the universal enveloping algebras of anti-Jordan triple systems are not necessary to be finite-dimensional.
\end{abstract}



\maketitle


\section{Introduction}

Anti-Jordan triple systems were introduced by Faulkner and Ferrar in \cite{F}.
The classification of finite-dimensional simple anti-Jordan triple systems over an algebraically closed field
of characteristic $0$ was given by Bashir \cite[Theorem 6]{simen}.

\begin{definition}\cite{simen}
A vector space $V$ over a field $F$ of characteristic $\neq 2$ endowed with a trilinear operation $ V \times V\times V \to V$, $(a, b, c)\to \langle abc\rangle$
is said to be an \emph{{anti-Jordan triple system}} if the following conditions are fulfilled for all $a,b,c,d,e \in V$:
\[
\langle abc \rangle = - \langle cba\rangle,
\qquad
\langle ab\langle cde\rangle\rangle = \langle\langle abc\rangle de\rangle + \langle c\langle bad\rangle e\rangle +\langle cd\langle abe\rangle\rangle.
\]
\end{definition}

If $A$ is an associative algebra, $A$ defines an anti-Jordan triple system $A_{-}$ relative to the
product $\langle abc\rangle = abc - cba$.

\begin{definition}
\emph{A representation} of an anti-Jordan triple system $\J$ is a homomorphism
$\rho\colon \J \to ( \mathrm{End}\,V )_{-}$ from $\J$ to the anti-Jordan triple system of endomorphisms of a vector space $V$.
In other words, $\rho$  is a linear mapping that satisfies
  \[
  \rho(\langle abc\rangle) = \rho(a)\rho(b)\rho(c) - \rho(c)\rho(b)\rho(a),
  \]
for all $a,b,c\in\J$.
Two representations $\rho_1$ and $\rho_2$ of an anti-Jordan triple system $\J$ on the same vector space $V$
are \emph{equivalent} if there exists an invertible endomorphism $T$ such that $\rho_2(a) = T^{-1} \rho_1(a) T$ for all  $a \in \J$.
\end{definition}

In this paper we use the theory of non-commutative Gr\"obner bases to prove that the universal enveloping algebra
of the simple anti-Jordan triple system of all $n \times n$ matrices is finite-dimensional.
This theory was used by Bergman \cite{Bergman} to give a new proof
of the PBW theorem and was used recently by Elgendy\cite{Elgendy2} and Elgendy and Bremner\cite{ElgendyBremner} to construct universal associative envelopes of nonassociative triple systems and universal envelopes of the $(n{+}1)$-dimensional $n$-Lie algebras respectively.

This paper is structured as follows. In Section 2, we recall basic results on noncommutative Gr\"obner bases.
In Section 3, we prove that the universal enveloping algebra of the simple anti-Jordan triple system of $n \times n$ matrices
over an algebraically closed field is finite-dimensional using Gr\"obner bases in free associative algebras.
In Section 4, we determine the structure constants of the universal enveloping algebra.
In Section 5, we determine the center of the universal enveloping algebra.
In the last section, we explicitly determine the complete decomposition of the universal enveloping algebra into a direct sum of matrix algebras. We also provide an example of a non-simple anti-Jordan triple system with infinite dimensional envelope. For examples of simple anti-Jordan triple systems with infinite dimensional envelopes see \cite{Elgendy, M}.

Unless otherwise stated, we assume throughout that all vector spaces are over
an algebraically closed field $F$ of characteristic 0.


\section{Preliminaries}

In this section we recall the basic definitions and results in the theory of
noncommutative Gr\"obner bases in free associative algebras following
\cite{deGraaf, ElgendyBremner}.

\begin{definition}
Let $X = \{ x_1, \dots, x_{n} \}$ be a set of symbols with the total order
$x_i < x_j$ if and only if $i < j$. The \emph{free monoid} generated by $X$ is
the set $X^\ast$ of all (possibly empty) words $w = x_{i_1} \cdots x_{i_k}$ ($k
\ge 0$) with the (associative) operation of concatenation. For $w = x_{i_1}
\cdots x_{i_k} \in X^*$ the \emph{degree} is $\deg(w) = k$.
The \emph{degree-lexicographical} (\emph{deglex})
order $<$ on $X^\ast$ is defined as follows: $u < v$ if and only if either ($i$)
$\deg(u) < \deg(v)$ or ($ii$) $\deg(u) = \deg(v)$ and $u = w x_i u'$, $v = w
x_j v'$ where $x_i < x_j$ ($w, u', v' \in X^\ast$). The \emph{free (unital) associative algebra} generated by $X$ is the vector space $F \langle X
\rangle$ with basis $X^*$ and multiplication extended bilinearly from
concatenation in $X^\ast$.
\end{definition}

\begin{definition}
The \emph{support} of a noncommutative polynomial $f \in F\langle X \rangle$ is
the set of all monomials $w \in X^\ast$ that occur in $f$ with nonzero
coefficient. The \emph{leading monomial} of $f \in F\langle X\rangle$, denoted
$\mathrm{LM}(f)$, is the highest element of the support of $f$ with respect to
deglex order. If $I$ is any ideal of $F\langle X \rangle$ then the set of
\emph{normal words} modulo $I$ is defined by $N(I) = \{ \, u \in X^* \mid
\text{$u \ne \mathrm{LM}(f)$ for any $f \in I$} \, \}$. We write $C(I)$ for the
subspace of $F\langle X \rangle$ spanned by $N(I)$.
\end{definition}

\begin{proposition} \label{C(I)proposition}
If $I \subset F\langle  X\rangle$ is an ideal then $F\langle X \rangle = C(I)
\oplus I$.
\end{proposition}

\begin{definition}
Let $G \subset F\langle X \rangle$ be a subset generating an ideal $I
\subset F\langle X \rangle$. A noncommutative polynomial $f \in F\langle X
\rangle$ is in \emph{normal form modulo} $G$ if no monomial occurring in $f$
has a factor of the form $\mathrm{LM}(g)$ for any $g \in G$.
A subset $G \subset I$
is a \emph{Gr\"obner basis} of $I$ if for all $f \in I$ there is a $g \in G$
such that $\mathrm{LM}(g)$ is a factor of $\mathrm{LM}(f)$.
A subset $G \subset F\langle X \rangle$ is \emph{self-reduced} if every $g
\in G$ is in normal form modulo $G \setminus \{g\}$ and every $g \in G$ is
\emph{monic}: the coefficient of $\mathrm{LM}(g)$ is 1.
\end{definition}

\begin{definition}
Let $g, h \in F \langle X \rangle$ be two monic noncommutative polynomials.
Assume that $\mathrm{LM}(g)$ is not a factor of $\mathrm{LM}(h)$ and that
$\mathrm{LM}(h)$ is not a factor of $\mathrm{LM}(g)$. Let $u, v \in X^\ast$ be
such that
$\mathrm{LM}(g)\,u = v\,\mathrm{LM}(h)$,
$u$ is a proper right factor of $\mathrm{LM}(h)$,
and
$v$ is a proper left factor of $\mathrm{LM}(g)$.
In this case the element $g u - v h \in F \langle X \rangle$ is called a
\emph{composition} of $g$ and $h$.
\end{definition}

\begin{theorem} \label{di}
If $I \subset F\langle X \rangle$ is an ideal generated by a self-reduced set
$G$, then $G$ is a Gr\"obner basis of $I$ if and only if for all compositions
$f$ of the elements of $G$ the normal form of $f$ modulo $G$ is zero.
\end{theorem}


\section{The universal associative enveloping algebra }

Let $\J$ be the anti-Jordan triple system of all $n \times n$ matrices over an algebraically closed field $F$ of characteristic 0 with triple product
$\langle a,b,c\rangle =  abc - cba.$

\begin{definition}
Let $ \Omega = \{ 1,2,\dots,n\}$ be a finite index set.
Let $B = \{E_{i,j}\}_{i,j \in \Omega}$ be an ordered basis of $\J$, where $E_{i,j}$ is the matrix with a single 1,
in the $i$th row and $j$th column, and zeros elsewhere. The structure constants for $\J$ are
\[
\langle E_{i,j} , E_{k,\ell}, E_{m,t}\rangle = \delta_{j,k} \delta_{\ell,m} E_{i,t} - \delta_{t,k} \delta_{\ell,i} E_{m,j},
\quad \text{for all } \,\, i,j,k,\ell,m,t\in \Omega.
\]
Consider the bijection $\phi\colon B \to X = \{ e_{i,j}\}_{i,j \in \Omega}$ defined by
$\phi(E_{i,j}) = e_{i,j}$. We extend $\phi$ to a linear map $\phi\colon \J \to F \langle
X \rangle$.
Throughout this paper we use the deglex order $<$ where $e_{i,j} < e_{k,\ell}$ if either $i<k$, or $i=k$ and $j<\ell$.
\end{definition}

\begin{definition}\label{1}
Let $G \subset F \langle X \rangle$ consist of these elements ($i, j, k, r, s, t \in \Omega$):
\allowdisplaybreaks
\begin{align*}
\mathcal{R}_1^{(i,j,k,t)} &= e_{i,j} e_{j,k} e_{k,t} - e_{k,t} e_{j,k} e_{i,j}- e_{i,t}\quad (k<i),
\\
\mathcal{R}_2^{(i,j,t)} &= e_{i,j} e_{j,i} e_{i,t} - e_{i,t} e_{j,i} e_{i,j}- e_{i,t}\quad  (t<j),
\\
\mathcal{R}_3^{(i,j,k,t)} &= e_{i,j} e_{k,i} e_{t,k} - e_{t,k} e_{k,i} e_{i,j}+ e_{t,j}\quad (t<i),
\\
\mathcal{R}_4^{(i, j,k)} &= e_{i,j} e_{k,i} e_{i,k} - e_{i,k} e_{k,i} e_{i,j}+ e_{i,j}\quad (k<j),
\\
\mathcal{R}_5^{(i,j,k,t,r,s)} &= e_{i,j}e_{k,t}e_{r,s} - e_{r,s} e_{k,t} e_{i,j}
\quad
\left(r<i,\, j \neq k \,\,\text{or}\,\, t \neq r, \, s \neq k \,\,\text{or}\,\, t \neq i \right),
\\
\mathcal{R}_6^{(i,j,k,t,s)} &= e_{i,j}e_{k,t}e_{i,s} - e_{i,s} e_{k,t} e_{i,j}
\quad
\left(s<j,\, j\neq k \,\,\text{or}\,\, t \neq i,\, s\neq k\,\, \text{or}\,\, t \neq i \right).
\end{align*}
Let $I \subset F\langle X \rangle$ be the ideal generated by $G$. We write $\mathfrak{A} = F\langle X \rangle / I$ with surjection $\pi\colon
F\langle X \rangle \to \mathfrak{A}$ sending $f$ to $f + I$, and $i = \pi \circ \phi$ for
the natural map $i \colon \J \to \mathfrak{A}$.
\end{definition}

\begin{lemma}
The unital associative algebra $\mathfrak{A}$ and the linear map $i$ form the universal
associative envelope of the anti-Jordan triple system $\J$.
\end{lemma}

Our goal in the rest of this section is to derive a Gr\"obner basis for the ideal $I$ from the set $G$ of generators.
This will be achieved by repeatedly calculating normal forms of compositions of generators.

\begin{definition}
We write $\delta_{i,j}$ for the Kronecker delta, and $\widehat{\delta}_{i,j} = 1 - \delta_{i,j}$.
\end{definition}

\begin{lemma}\label{first}
The set of all normal forms modulo $G$ of nontrivial compositions among elements of $G$ includes the set $G_1$ which consists of the elements:
\begin{align*}
\mathcal{G}^{(r,t,m)}_{1} &= e_{r,t} e_{t,m} - e_{r,1} e_{1,m} \quad  ( m\neq r,\, t\neq 1),
\\
\mathcal{G}^{(i,t,\ell)}_{2} &= e_{i,t} e_{\ell,i} - e_{1,t} e_{\ell,1}\quad  (t \neq \ell,\,i\neq 1),
\\
\mathcal{G}^{(i,j,k,\ell)}_{3} &= e_{i,j} e_{k,\ell}\quad  (i \neq \ell,\, j\neq k).
\end{align*}
\end{lemma}
\begin{proof}
For all $s<t$, we consider the following composition:
\[
S = \mathcal{R}_2^{(r,t,s)} e_{s,m} - e_{r,t} \mathcal{R}_1^{(t,r,s,m)}.
\]
We eliminate from $S$ all occurrences of the leading monomials of $G$ as factors in the monomials;
we write $\equiv$ to indicate congruence modulo $G$:
\begin{dmath*}
  S  =  - e_{r,s} e_{t,r} e_{r,t}e_{s,m} - e_{r,s}e_{s,m}+ e_{r,t} e_{s,m} e_{r,s} e_{t,r} + e_{r,t}e_{t,m}
 \equiv - e_{r,s} \left( e_{s,m}e_{r,t} e_{t,r}- \delta_{m,r} e_{s,r}\right)- e_{r,s}e_{s,m}+\left(e_{r,s} e_{s,m}e_{r,t}   - \delta_{m,r}e_{r,t}\right)e_{t,r}+ e_{r,t} e_{t,m}
 = \delta_{m,r} e_{r,s}  e_{s,r}- e_{r,s}e_{s,m}- \delta_{m,r}e_{r,t}e_{t,r}+ e_{r,t} e_{t,m} ,
\end{dmath*}
using the relations $\mathcal{R}_3^{(t,r,r,s)}$, $\mathcal{R}_5^{(t,r,r,t,s,m)}$,  $\mathcal{R}_4^{(r,t,s)}$ and $\mathcal{R}_6^{(r,t,s,m,s)}$.
Clearly, if  $m = r$ then  $S \equiv 0$.  Assume $m\neq r$ and obtain the set $L$ of nonzero normal forms of $S$ modulo $G$:
\[
L= \big\{ \, \mathcal{N}^{(r,t,m,s)} =  e_{r,t} e_{t,m} - e_{r,s}e_{s,m}\mid \text{for all}\,\,s< t,  \, m\neq r \, \big\}.
\]
The set $L$ is not self-reduced. Therefore, for all $ 1<s<t \leq n$,
we eliminate from the element $\mathcal{N}^{(r,t,m,s)}$ occurrence of the leading monomial of $\mathcal{N}^{(r,s,m,1)}$
and obtain a self-reduced set consisting of the elements $\mathcal{G}^{(r,t,m)}_{1}$.

For all $ (r, \ell)< (i,k)$, we consider the following composition:
\begin{align*}
 S_1 & = \mathcal{R}_1^{(i,k,r,t)} e_{\ell,s} - e_{i,k}\mathcal{R}_5^{(k,r,r,t,\ell,s)}
\quad ( t\neq \ell, \,\, \text{and} \,\, s \neq r \,\, \text{or}   \,\, t\neq k ).
\end{align*}
We eliminate from $S_1$ all occurrences of the leading monomials of elements of $G$:
\begin{align*}
S_1 &= - e_{r,t} e_{k,r} e_{i,k} e_{\ell ,s}- e_{i,t} e_{\ell ,s}
 +e_{i,k} e_{\ell ,s} e_{r,t} e_{k,r}
 \\
 & \equiv - e_{r,t} (e_{\ell ,s} e_{i,k} e_{k,r}  - \delta_{s,i} e_{\ell,r}) - e_{i,t} e_{\ell ,s}+ e_{r,t}  e_{\ell ,s} e_{i,k} e_{k,r}
 \\
 & = \delta_{s,i} e_{r,t}e_{\ell,r} - e_{i,t} e_{\ell ,s},
\end{align*}
\allowdisplaybreaks
using the relations $\mathcal{R}_3^{(k,r,i,\ell)} $, $\mathcal{R}_5^{(k,r,i,k,\ell,s)}$ and $\mathcal{R}_5^{(i,k,\ell,s,r,t)} $.
Hence, for all $ (r,\ell)<(i,k) $, the possible (monic) normal forms of $S_1$ are
\begin{align}
  e_{i,t} e_{\ell,s}\,\,  (\text{if}\,\, i\neq s),
  \qquad
  \qquad
  e_{i,t} e_{\ell,i} - e_{r,t} e_{\ell,r}\,\, (\text{if}\,\, i = s).\label{nf1}
  \end{align}
For all $(r,t)< (i,k)$, we consider the following composition:
\begin{dmath*}S_2 = \mathcal{R}_5^{(i,j,k,\ell,r,s)} e_{t,m} - e_{i,j} \mathcal{R}_5^{(k,\ell,r,s,t,m)}
\\
{( s\neq k\,\, \text{or}\,\, \ell\neq i , j\neq k\,\, \text{or}\,\,\ell \neq r,  m\neq r\,\, \text{or}\,\, s\neq k )\,\, \text{and}\,\, (\ell\neq r\,\, \text{or}\,\,s \neq t)}.
\end{dmath*}
We eliminate from $S_2$ all occurrences of the leading monomials of elements of $G$:
\begin{dmath*}
 S_2 =  - e_{r,s} e_{k,\ell} e_{i,j}e_{t,m} +  e_{i,j}  e_{t,m} e_{r,s} e_{k,\ell},
 \equiv -e_{r,s}\left( e_{t,m} e_{i,j} e_{k,\ell} + \delta_{\ell,i} \delta_{j,t} e_{k,m} - \delta_{j,k} \delta_{m,i} e_{t,\ell}\right) + \left(e_{r,s} e_{t,m} e_{i,j}+ \delta_{j,t} \delta_{m,r} e_{i,s}-\delta_{s,t} \delta_{m,i} e_{r,j}\right) e_{k,\ell}
 = \delta_{j,t}\left( - \delta_{\ell,i} e_{r,s} e_{k,m} + \delta_{m,r} e_{i,s}e_{k,\ell}\right) - \delta_{m,i} \left( -\delta_{j,k} e_{r,s} e_{t,\ell} + \delta_{s,t} e_{r,j} e_{k,\ell}\right),
   \end{dmath*}
using the relations $\mathcal{R}_1^{(k,\ell,j,m)}\!\!, \mathcal{R}_3^{(k,\ell,m,t)}\!\!,\mathcal{R}_5^{(k,\ell,i,j,t,m)}\!\!, \mathcal{R}_1^{(i,j,r,s)}\!\!,  \mathcal{R}_3^{(i,j,t,r)}\!$ and $\mathcal{R}_5^{(i,j,t,m,r,s)}$.
We first note that if $(m, j, s) = (i, k, t)$ then the (monic) normal form of $S_2$ modulo $G$ coincides with the element $\mathcal{N}^{(r,k,\ell,t)}$,
so we ignore this case.
For all $(r,t)<(i,k)$, the possible non-zero (monic) normal forms of $S_2$ modulo $G$ are
\begin{alignat}{2} \label{nf2}
&
e_{r,s} e_{k,m} \;\;   (m\neq r, s\neq k) , &\qquad
&
e_{i,s} e_{k,\ell}\;\;  (\ell \neq i, s\neq k),
\notag
\\
&
e_{r,s} e_{t,\ell}\;\;  (r\neq \ell, s\neq t), &\qquad
&
e_{r,j} e_{k,\ell}\;\;  (\ell\neq r,j \neq  k),
\\
&
e_{r,t} e_{k,i}\;\;  (t\neq k,\, r < i), &\qquad
&
e_{i,s} e_{k,i}- e_{r,s} e_{k,r}\quad  ( s\neq k).
\notag
\end{alignat}
Combining \eqref{nf1} and \eqref{nf2} gives all the possible normal forms of $S_1$ and $S_2$:
\begin{align*}
& \mathcal{L}^{(i,s,k,r)}= e_{i,s}e_{k,i}- e_{r,s} e_{k,r}\quad  (r<i,\, s\neq k), 
\\
&\mathcal{G}_3^{(i,j,k,\ell)} = e_{i,j} e_{k,\ell}\quad (i\neq k,\, j\neq \ell).\end{align*}
We observe that the set
$\big\{\mathcal{L}^{(i,s,k,r)}\mid \text{for all}\,\,r< i,  \, s\neq k \big\}$ is not self-reduced. Therefore, for all $ 1<r< i\leq n$,
we eliminate from  the element $\mathcal{L}^{(i,s,k,r)}$ occurrence of the leading monomial of
$\mathcal{L}^{(r,s,k,1)}$ and obtain a self-reduced set consisting of the elements $\mathcal{G}_2^{(i,s,k)}$.
For $n = 2$, we cannot obtain $ \mathcal{G}^{(1,2,1,2)}_{3}$,  $ \mathcal{G}^{(2,2,1,1)}_{3}$ and $ \mathcal{G}^{(2,1,2,1)}_{3}$
from $S_1$ or $S_2$. Thus, we consider three more compositions in this case:
\begin{align*}
& S_3 = \mathcal{R}_6^{(1,2,1,2,1)} e_{1,1}- e_{1,2}\mathcal{R}_4^{(1,2,1)},\qquad S_4 = \mathcal{R}_1^{(2,2,1,2)}e_{1,1}-e_{2,2} \mathcal{R}_3^{(2,1,1,1)},
\\
& S_5 = \mathcal{R}_2^{(2,2,1)} e_{2,1}- e_{2,2}\mathcal{R}_6^{(2,2,2,1,1)}.
\end{align*}
We eliminate from $S_3$ all the leading monomials of elements of $G$ and obtain
\begin{dmath*}
S_3= - e_{1,1} e_{1,2}^2 e_{1,1} + e_{1,2} e_{1,1}^2 e_{1,2} - e_{1,2}^2
  \equiv {-e_{1,1}\left(e_{1,1} e_{1,2}^2\right)+ \left(e_{1,1}^2 e_{1,2}- e_{1,2}\right)e_{1,2}- e_{1,2}e_{1,2}
 = -2 e_{1,2}^2,}
\end{dmath*}
using the relations $\mathcal{R}_6^{(1,2,1,2,1)}$ and $\mathcal{R}_4^{(1,2,1)}$.
Similarly, we can show that $ S_4 \equiv -2e_{2,2} e_{1,1}$ and $ S_5 \equiv -2 e_{2,1}e_{2,1}$.
The monic forms of the last three elements give the required elements. This completes the proof.
\end{proof}

\begin{lemma}\label{second}
The set of all normal forms modulo $G \cup G_1$ of nontrivial compositions among elements of $G \cup G_1$ includes the set $G_2$ which consists of the elements:
\begin{align*}
 & \mathcal{G}^{(r, i)}_{4} = e_{r, i} e_{i,r} - e_{r,1} e_{1,r} + e^2_{1,1} - e_{1,i} e_{i,1} \quad  ( r,i \in \Omega \setminus \{1\}).
\end{align*}
\end{lemma}

\begin{proof}
For all $(s, t) < (r,i)$, we consider the following composition:
\[
S = \mathcal{R}_1^{(r,i,s,t)}e_{t,m} - e_{r,i} \mathcal{R}_1^{(i,s, t,m)}.
\]
We eliminate from $S$ all occurrences of the leading monomials of elements of $G$:
\begin{dmath*}
S
= - e_{s,t} e_{i,s} e_{r,i}e_{t,m} - e_{r,t} e_{t,m} + e_{r,i} e_{t,m} e_{s,t} e_{i,s}+ e_{r,i}e_{i,m}
\equiv
 -e_{s,t} \left( e_{t,m} e_{r,i} e_{i,s} - \delta_{m,r} e_{t,s}\right)- e_{r,t} e_{t,m}+
 \left(e_{s,t} e_{t,m} e_{r,i} - \delta_{m,r} e_{s,i}\right)e_{i,s}+ e_{r,i}e_{i,m}
=
\delta_{m,r} e_{s,t} e_{t,s}  - e_{r,t} e_{t,m}- \delta_{m,r} e_{s,i}e_{i,s} +e_{r,i}e_{i,m},
\end{dmath*}
using the relations $\mathcal{R}_3^{(i,s,r,t)}$, $\mathcal{R}_5^{(i,s,r,i,t,m)}$, $\mathcal{R}_3^{(r,i,t,s)}$ and $\mathcal{R}_5^{(r,i,t,m,s,t)}$.
We now eliminate from $S$ all occurrences of the leading monomials of elements of $G_1$.
Clearly, if $m \neq r$ then $S\equiv 0\!\!\mod{G_1}$, using the relations $\mathcal{G}^{(r,t,m)}_1$(if $t\neq $1) and $\mathcal{G}^{(r,i,m)}_1$.
Assume $m = r$ and obtain the set $\mathcal{N}$ of nonzero normal forms of $S$ modulo $G\cup G_1$:
  \[\mathcal{N} = \big\{ \mathcal{N}^{(r,i,t,s)} =  e_{r,i}e_{i,r} - e_{r,t} e_{t,r} - e_{s,i} e_{i,s} +e_{s,t} e_{t,s}\mid \text{for all} \,  (s,t)<(r,i) \big\}.\]
We observe that the set $\mathcal{N}$ is not self-reduced and the element $\mathcal{N}^{(r,i,1,1)}$ coincides with $\mathcal{G}^{(r,i)}_4$ for all $r,i \neq 1$.
Assume now that $s, t\neq 1 $. For all $(s,t)<(r,i)$, we eliminate from $\mathcal{N}^{(r,i,t,s)}$ occurrence of the leading monomials of
$\mathcal{N}^{(r,t,1,1)}$, $\mathcal{N}^{(s,i,1,1)}$ and $\mathcal{N}^{(s,t,1,1)}$ and again obtain  $\mathcal{G}^{(r,i)}_4$.
A similar argument can be used if $s\neq 1$ or $t\neq 1$. The result is a self-reduced set consisting of the elements $\mathcal{G}^{(r,i)}_4$.
\end{proof}
\allowdisplaybreaks
\begin{lemma}\label{GG3}
The set of all normal forms modulo $G \cup G_1 \cup G_2$ of nontrivial compositions among elements of $G \cup G_1 \cup G_2$ includes the set $G_3$ which consists of the elements:
\begin{align*}
\mathcal{G}^{(r,i)}_5 &= e_{r,1} e_{1,i} e_{i,1} - e_{1,1}^2 e_{r,1} - e_{r,1}\quad (r < i;\, i, r \in \Omega \setminus \{1\}),
\\
\mathcal{G}^{(i,r)}_{6} &= e_{i,1} e_{1,i} e_{r,1} - e_{1,1}^2 e_{r,1} \quad ( i<r;\,    i,r \in \Omega \setminus \{1\} ),
\\
\mathcal{G}^{(t,\ell)}_{7}  &= e_{1,t} e_{t,1} e_{1,\ell} - e_{1,1}^2 e_{1,\ell} \quad ( t<\ell;\, t,\ell \in  \Omega \setminus \{1\} ),
\\
\mathcal{G}^{(\ell,t)}_{8}  &= e_{1,\ell} e_{t,1} e_{1,t}-e_{1,1}^2 e_{1,\ell} + e_{1,\ell} \quad (\ell < t; \ell,t \in \Omega \setminus \{1\}),
\\
\mathcal{G}^{(r)}_9  &=  e_{r,1} e_{1,r} e_{r,1} -  2e_{1,1}^2 e_{r,1} - e_{r,1}\quad (r \in \Omega \setminus \{1\}),
\\
\mathcal{G}^{(r)}_{10}  &=  e_{1,r} e_{r,1} e_{1,r} -  2e_{1,1}^2 e_{1,r} + e_{1,r}\quad (r \in \Omega \setminus \{1\}),
\\
\mathcal{G}^{(r,i,\ell)}_{11}  &= e_{r,1} e_{1,i} e_{\ell,1}\quad (r\neq  i\neq \ell),
\\
\mathcal{G}^{(\ell,i,r)}_{12}  &= e_{1,\ell} e_{i,1} e_{1,r}\quad (r\neq  i\neq \ell),
\\
\mathcal{G}^{(i)}_{13} &= e_{1,1} e_{1,i}e_{i,1}   - \tfrac{1}{2} e_{1,1}^3 -\tfrac{1}{2}  e_{1,1}\quad (i \in \Omega \setminus \{1\}),
\\
\mathcal{G}^{(i)}_{14}  &= e_{1,1} e_{i,1} e_{1,i} - \tfrac{1}{2}e_{1,1}^3 +\tfrac{1}{2} e_{1,1}\quad (i \in \Omega \setminus \{1\}).
\end{align*}
\end{lemma}
\begin{proof}
For all $r,t,i,\ell ,k\in \Omega$, we consider the following six compositions:
  \begin{align*}
    S_1 &= \mathcal{G}^{(r,t,i)}_{1}e_{i,\ell} - e_{r,t} \mathcal{G}^{(t,i,\ell)}_{1}\quad (1\neq i \neq r,\,   \ell \neq t \neq 1),
    \\
    S_2& = \mathcal{G}^{(r,i)}_4 e_{r,t} - e_{r,i} \mathcal{G}^{(i,r,t)}_1 \quad (i,r \neq 1,\,i\neq t),
    \\
    S_3& = \mathcal{G}^{(i,r)}_4 e_{t,r} - e_{i,r} \mathcal{G}^{(r,i,t)}_2 \quad (i,r \neq 1,\,i\neq t),
    \\
    S_4&= \mathcal{G}^{(r,t,i)}_{1} e_{\ell,k} - e_{r,t} \mathcal{G}_{3}^{(t,i,\ell,k)}\quad (r\neq i\neq \ell, k\neq t\neq 1),
    \\
    S_{5}& = \mathcal{G}^{(t,\ell,i)}_{2} e_{k,r}- e_{t,\ell} \mathcal{G}_{3}^{(i,t,k,r)}\quad (\ell\neq i\neq r, k\neq t\neq 1),
        \\
    S_6 &= \mathcal{G}_4^{(r,i)} e_{t,\ell} - e_{r,i} \mathcal{G}_3^{(i,r,t,\ell)}\quad ( i,r \neq 1 , \ell \neq i, r\neq t).
    \end{align*}
We eliminate from these compositions all occurrences of the leading monomials of $G \cup G_1 \cup G_2$. For the composition $S_1$, we have
\begin{align*}
S_1 &= - e_{r,1} e_{1,i} e_{i,\ell} + e_{r,t} e_{t,1} e_{1,\ell}
\\
&\equiv - \delta_{\ell,1} e_{r,1} e_{1,i} e_{i,1} - \widehat{\delta}_{\ell,1} e_{r,1} e_{1,1} e_{1,\ell}+ \delta_{r,1} e_{1,t} e_{t,1} e_{1,\ell} +  \widehat{\delta}_{r,1} e_{r,1} e_{1,1} e_{1,\ell}\!\!\!\!\mod{G_1},
\end{align*}
using the relations $\mathcal{G}^{(1,i,\ell)}_{1}$ and $\mathcal{G}^{(r,t,1)}_{1}$.
We note first that if $\ell, r \neq 1$ then $S \equiv 0 \!\!\mod{G_1}$. Three cases need to be considered.
Case I. If $(\ell,  r) = (1, 1)$ then
\begin{align*}
S_1& \equiv - e_{1,1} e_{1,i} e_{i,1} + e_{1,t} e_{t,1} e_{1,1}\!\!\!\!\mod{G_1}
\\& \equiv - e_{1,1} e_{1,i} e_{i,1} + e_{1,1} e_{t,1} e_{1,t} + e_{1,1}  \!\!\!\!\mod{G},
\end{align*}
using the relation $\mathcal{R}^{(1,t,1)}_{2}$, since by definition $ t\neq 1$. Hence the (monic) normal form of $S_1$ in this case is
\begin{equation}\label{pr1} \mathcal{G}{^\prime}^{(t,i)}  =  e_{1,1} e_{t,1} e_{1,t} - e_{1,1} e_{1,i} e_{i,1}  + e_{1,1}\quad (t,i \in \Omega \setminus \{1\}).\end{equation}
Case II. If $\ell =1$ and $r\neq 1$ then
\begin{align*}
S_1 &\equiv - e_{r,1} e_{1,i} e_{i,1} + e_{r,1} e_{1,1}^2\!\!\!\!\mod{G_1}
 \\&\equiv - e_{r,1} e_{1,i} e_{i,1} + e_{1,1}^2 e_{r,1} + e_{r,1}\!\!\!\!\mod{G},
\end{align*}
using the relation $\mathcal{R}^{(r,1,1,1)}_{1}$. Clearly, if $r< i$ then the monic form of the last equation coincides with $\mathcal{G}^{(r,i)}_5$.
If $i<r$ then the  element $e_{r,1} e_{1,i} e_{i,1}$ of the last equation can be reduced further modulo $G$:
$ e_{r,1} e_{1,i} e_{i,1}\equiv e_{i,1} e_{1,i} e_{r,1}+e_{r,1} \!\!\mod{G}$.
Using this in the last equation gives $\mathcal{G}^{(i,r)}_{6}$.  Case III. If $\ell \neq 1$ and $r =1$ then
\begin{align*}
S_1 &\equiv - e_{1,1}^2 e_{1,\ell} + e_{1,t} e_{t,1} e_{1,\ell}\!\!\!\!\mod{G_1}.
\end{align*}
Clearly, if $t<\ell$ then the normal form of $S_1$ in this case coincides with $\mathcal{G}^{(t,\ell)}_{7}$.
If $ \ell <t$ then the element $e_{1,t} e_{t,1} e_{1,\ell}$ of the last equation can be reduced further modulo $G$:
$e_{1,t} e_{t,1} e_{1,\ell} \equiv e_{1,\ell} e_{t,1} e_{1,t}+e_{1,\ell} \!\!\mod{G}$. Using this in the last equation gives $\mathcal{G}_{8}^{(\ell,t)}$.

For the composition $S_2$, we have
  \begin{dmath*}
 S_2 = - e_{r,1} e_{1,r} e_{r,t} + e_{1,1}^2 e_{r,t} - e_{1,i} e_{i,1} e_{r,t} + e_{r,i}e_{i,1}e_{1,t}
  \equiv - \delta_{t,1} e_{r,1} e_{1,r} e_{r,1} - \widehat{\delta}_{t,1} e_{r,1} e_{1,1} e_{1,t} + \delta_{t,1} e_{1,1}^2 e_{r,1} + e_{r,1} e_{1,1} e_{1,t}\!\!\!\!\mod{G_1}
 \equiv - \delta_{t,1} e_{r,1} e_{1,r} e_{r,1}  + \delta_{t,1} \left(2e_{1,1}^2 e_{r,1}+e_{r,1}\right)\!\!\!\!\mod{G},
  \end{dmath*}
using the relations $\mathcal{G}^{(1,r,t)}_{1}$, $\mathcal{G}^{(1,1,r,t)}_{3}$, $\mathcal{G}^{(i,1,r,t)}_{3}$, $\mathcal{G}^{(r,i,1)}_{1}$ and $\mathcal{R}^{(r,1,1,t)}_{1}$.
Hence, for $t=1$ the (monic) normal form of $S_2$ coincides with $\mathcal{G}^{(r)}_9$.
 For the composition $S_3$, we have
 \begin{dmath*}
S_3 = -e_{i,1} e_{1,i} e_{t,r} + e_{1,1}^2 e_{t,r} - e_{1,r} e_{r,1} e_{t,r} + e_{i,r}e_{1,i} e_{t,1}
\equiv \delta_{t,1} e_{1,1}^2 e_{1,r} - \widehat{\delta}_{t,1} e_{1,r} e_{1,1}e_{t,1}- \delta_{t,1} e_{1,r} e_{r,1}e_{1,r} + e_{1,r}e_{1,1} e_{t,1}\!\!\!\!\mod{G_1}
\equiv \delta_{t,1} \left(e_{1,1}^2 e_{1,r} - e_{1,r} e_{r,1}e_{1,r} + e_{1,1}^2 e_{1,r} -e_{1,r} \right)\!\!\!\!\mod{G},
  \end{dmath*}
using the relations $\mathcal{G}^{(1,i,t,r)}_{3}$, $\mathcal{G}^{(1,1,t,r)}_{3}$, $\mathcal{G}^{(i,r,1)}_{2}$,  $\mathcal{G}^{(r,1,t)}_{2}$ and $\mathcal{R}^{(1,r,1)}_{4}$.
Hence, for $t=1$ the (monic) normal form of $S_3$ coincides with $\mathcal{G}^{(r)}_{10}$. Next, we consider the composition $S_4$:
 \[S_4 = -e_{r,1} e_{1,i} e_{\ell,k} \equiv - \delta_{k,1} e_{r,1} e_{1,i} e_{\ell,1}\!\!\!\!\mod{G_1},\]
 using the relation $\mathcal{G}^{(1,i,\ell,k)}_{3}$.
Obviously, for $k = 1$ the (monic) normal form of $S_4$ coincides with $\mathcal{G}^{(r,i,
\ell)}_{11}$. Similarly, we can show that for $k =1$, the (monic) non zero normal form of $S_5$ coincides with   $\mathcal{G}^{(\ell,i,r)}_{12}.$ Finally, for the composition $S_6$, we have
 \begin{align*}
 S_6 &= - e_{r,1} e_{1,r} e_{t,\ell} + e_{1,1}^2 e_{t,\ell} - e_{1,i} e_{i,1} e_{t,\ell}
 \\
 &\equiv  - \delta_{\ell,1} e_{r,1} e_{1,r} e_{t,1} + \delta_{t,1} e_{1,1}^2 e_{1,\ell} + \widehat{\delta}_{t,1} \delta_{\ell,1} e_{1,1}^2 e_{t,1} - \delta_{t,1} e_{1,i} e_{i,1} e_{1,\ell} \!\!\!\!\mod{G_1},
 \end{align*}
 using the relations $\mathcal{G}^{(1,r,t,\ell)}_{3}$, $\mathcal{G}^{(1,1,t,\ell)}_{3}$ and $\mathcal{G}^{(i,1,t,\ell)}_{3}$.
  Clearly, if $ \ell = 1$ and $ t\neq 1$ then the (monic) normal form of $S_6$ coincides with $\mathcal{G}^{(r,t)}_6$ (if $r<t$) and $\mathcal{G}^{(t,r)}_5$ (if $t<r$). If $\ell \neq 1$ and $t = 1$ then the (monic) normal form of $S_6$ coincides with $\mathcal{G}^{(i,\ell)}_7$ (if $i<\ell$) and  $\mathcal{G}^{(\ell,i)}_8$ (if $\ell <i$). If $(\ell ,t)= (1, 1)$, since by definition $i, r\neq 1$, we have
 \begin{dmath*}
S_6  \equiv  - e_{r,1} e_{1,r} e_{1,1} + e_{1,1}^3  - e_{1,i} e_{i,1} e_{1,1} \!\!\!\!\mod{G_1}
  \equiv -\left(e_{1,1} e_{1,r} e_{r,1}-e_{1,1} \right) + e_{1,1}^3 - \left(e_{1,1} e_{i,1} e_{1,i} +e_{1,1}\right)\!\!\!\!\mod{G},
\end{dmath*}
using the relations $\mathcal{R}^{(r,1,1,1)}_{3}$ and $\mathcal{R}^{(1,i,1)}_{2}$.
Hence, the (monic) normal form of $S_6$ in this case is
 \begin{equation}\mathcal{G}^{{\prime\prime}^{(i,r)} }=   e_{1,1} e_{i,1} e_{1,i}+ e_{1,1} e_{1,r} e_{r,1}- e_{1,1}^3 \quad (i,r \in \Omega \setminus \{1\}).\label{pr}\end{equation}
 We note that the set $\mathcal{N} = \big\{ \, \mathcal{G}{^\prime}^{(t,i)},\, \mathcal{G}^{{\prime\prime}^{(i,r)} }|\,\, \text{for all}\,\, i, t,r \in \Omega \setminus \{1\} \, \big\}$ of the normal forms \eqref{pr1} and \eqref{pr} is not self-reduced. So, we eliminate from
$\mathcal{G}{^\prime}^{(i,i)}$ the leading monomial of  $\mathcal{G}^{{\prime\prime}^{(i,i)} }$ and obtain
\[\mathcal{G}{^\prime}^{(i,i)}= -2 e_{1,1}e_{1,i} e_{i,1} +e_{1,1}^3 +e_{1,1},\]
whose monic form coincides with $\mathcal{G}^{(i)}_{13}$.
We now eliminate from $\mathcal{G}^{{\prime\prime}^{(i,i)} }$ the leading monomial of $\mathcal{G}^{(i)}_{13}$ and obtain
\[ \mathcal{G}^{{\prime\prime}^{(i,i)} } = e_{1,1} e_{i,1} e_{1,i} + \tfrac{1}{2} e_{1,1}^3 + \tfrac{1}{2} e_{1,1} -e_{1,1}^3,\]
which coincides with $\mathcal{G}^{(i)}_{14}$. 
\end{proof}
\begin{lemma}\label{seventh}
The set of all normal forms modulo $G \cup G_1 \cup G_2 \cup G_3$ of nontrivial compositions among elements of $G \cup G_1 \cup G_2 \cup G_3$ includes the set $G_4$ which consists of the elements:
\begin{align*}
&\mathcal{G}^{(i)}_{17} = e_{1,1}^3 e_{1,i} - e_{1,1} e_{1,i},\quad
 \mathcal{G}^{(i)}_{18}=  e_{1,1}^3 e_{i,1} + e_{1,1} e_{i,1} \qquad (i\in \Omega \setminus \{1\}),
 \\
 &\mathcal{G}_{19}= e_{1,1}^5 - e_{1,1}.
\end{align*}
\end{lemma}

\begin{proof}
For all $i\in \Omega \setminus \{1\}$ we consider the following three compositions:
\begin{align*}
& S_1= \mathcal{G}^{(i)}_{14}e_{1,i}- e_{1,1} e_{i,1} \mathcal{G}^{(1,i,1,i)}_{3},\quad S_2 = \mathcal{G}^{(i)}_{13} e_{i,1}- e_{1,1}e_{1,i} \mathcal{G}^{(i,1,i,1)}_{3},
\\& S_3 =  \mathcal{G}_{13}^{(i)} e_{1,i} e_{i,1} -  e_{1,1} e_{1,i} \mathcal{G}_{9}^{(i)}.
\end{align*}
 We note that $S_1 = - \tfrac{1}{2} e_{1,1}^3 e_{1,i}+\tfrac{1}{2} e_{1,1}e_{1,i}$ and
 $S_2 = - \tfrac{1}{2} e_{1,1}^3e_{i,1} -\tfrac{1}{2} e_{1,1}e_{i,1}$ are in normal form modulo  $G \cup G_1 \cup G_2 \cup G_3$ and
the monic forms of $S_1$ and $S_2$ coincide with $\mathcal{G}^{(i)}_{17}$ and $\mathcal{G}^{(i)}_{18}$ respectively.
For the composition $S_3$, we have
\begin{align*}
 & S_3 =  -\tfrac{1}{2} e_{1,1}^3 e_{1,i} e_{i,1}-\tfrac{1}{2} e_{1,1} e_{1,i} e_{i,1}+2 e_{1,1} e_{1,i}e_{1,1}^2e_{i,1}+e_{1,1} e_{1,i} e_{i,1}
\\
& \equiv -\tfrac{1}{2} e_{1,1}^2\left( \tfrac{1}{2} e_{1,1}^3 + \tfrac{1}{2} e_{1,1}\right)+\tfrac{1}{2}\left(\tfrac{1}{2} e_{1,1}^3 + \tfrac{1}{2} e_{1,1}\right) \!\!\!\!\mod{G_3}
 = -\tfrac{1}{4} e_{1,1}^5 + \tfrac{1}{4} e_{1,1},
\end{align*}
 using the relations $\mathcal{G}^{(i)}_{13}$ and $\mathcal{G}^{(1,i,1)}_{11}$, whose monic form coincides with $\mathcal{G}_{19}$.
\end{proof}

\begin{lemma}\label{G*}
The self-reduced form $\mathfrak{G}$ of the set $G \cup G_1 \cup G_2 \cup G_3 \cup G_4$ consists of the elements:
\allowdisplaybreaks
\begin{align*}
\mathcal{G}_{0}^{(i,j)} &= e_{i,1} e_{1,1} e_{1,j} - e_{1,j} e_{1,1} e_{i,1} - e_{i,j}\quad ( i,j \in \Omega \setminus \{1\}),
\\
\mathcal{G}_{1}^{(i,j,k)} &= e_{i,j} e_{j,k} - e_{i,1} e_{1,k}  \quad  (i,j,k\in \Omega;\, k\neq i,  j \neq 1),
\\
 \mathcal{G}_{2}^{(i,j,k)} &= e_{i,j} e_{k,i} - e_{1,j} e_{k,1}\quad  (i,j,k \in \Omega;\, j\neq k,\, i \neq 1),
\\
 \mathcal{G}_{3}^{(i,j,k,\ell)}  &= e_{i,j} e_{k,\ell}\quad (i,j,k,\ell \in \Omega;\, i\neq \ell,\, j \neq k),
\\
  \mathcal{G}_{4}^{(i, j)}  &= e_{i, j} e_{j,i} - e_{i,1} e_{1,i}  - e_{1,j} e_{j,1}+ e^2_{1,1} \quad  (i, j\in \Omega \setminus \{1\}),
   \\
 \mathcal{G}_{5}^{(i,j)} &= e_{i,1} e_{1,j} e_{j,1} - e_{1,1}^2 e_{i,1} - e_{i,1}\quad ( i, j \in \Omega;\, i\neq 1, j\neq i),
 \\
  \mathcal{G}_{6}^{(i,j)} &= e_{j,1} e_{1,j} e_{i,1} - e_{1,1}^2 e_{i,1} \quad (i, j\in \Omega \setminus \{1\}; i\neq j),
 \\
  \mathcal{G}_{7}^{(i,j)} &= e_{1,i} e_{i,1} e_{1,j} - e_{1,1}^2 e_{1,j} \quad  (i, j \in \Omega \setminus \{1\};\, i \neq j),
    \\
 \mathcal{G}_{8}^{(i,j)} &= e_{1,i} e_{j,1} e_{1,j} - e_{1,1}^2 e_{1,i}+ e_{1,i}\quad (i,j\in \Omega;\, i\neq j, i\neq 1),
 \\
 \mathcal{G}_{9}^{(i)}  &= e_{i,1} e_{1,i}  e_{i,1}-2 e_{1,1}^2 e_{i,1} - e_{i,1}\quad (i \in \Omega \setminus \{1\}),
 \\
  \mathcal{G}_{10}^{(j)}  &=e_{1,j} e_{j,1} e_{1,j} - 2 e_{1,1}^2 e_{1,j} + e_{1,j} \quad  (j \in \Omega \setminus \{1\}) ,
 \\
 \mathcal{G}_{11}^{(i,j,k)}  &= e_{i,1} e_{1,j} e_{k,1}\quad ( k,i,j \in \Omega;\, k , i\neq j),
\\
\mathcal{G}_{12}^{(i,j,k)} &= e_{1,i} e_{j,1} e_{1,k} \quad  (k , i ,j \in \Omega;\,  i, k \neq j),
\\
\mathcal{G}_{13}^{(i)} &= e_{1,1} e_{1,i}e_{i,1}   - \tfrac{1}{2} e_{1,1}^3 -\tfrac{1}{2}  e_{1,1}\quad (i \in \Omega \setminus \{1\}),
\\
\mathcal{G}_{14}^{(i)}  &= e_{1,1} e_{i,1} e_{1,i} - \tfrac{1}{2}e_{1,1}^3 +\tfrac{1}{2} e_{1,1}\quad (i \in \Omega \setminus \{1\}),
\\
\mathcal{G}_{15}^{(i)}  &= e_{1,i} e_{i,1}e_{1,1} - \tfrac{1}{2}e_{1,1}^3 -\tfrac{1}{2}e_{1,1}\quad  (i \in \Omega \setminus \{1\}),
    \\
\mathcal{G}_{16}^{(i)}  &= e_{i,1} e_{1,i} e_{1,1}- \tfrac{1}{2}e_{1,1}^3 +\tfrac{1}{2} e_{1,1}\quad (i \in \Omega \setminus \{1\}),
 \\
\mathcal{G}_{17}^{(i)}  &= e_{1,1}^3 e_{1,i} - e_{1,1} e_{1,i} \quad (i \in \Omega \setminus \{1\}),
\\
 \mathcal{G}_{18}^{(i)} &=  e_{1,1}^3 e_{i,1} + e_{1,1} e_{i,1} \quad (i \in \Omega \setminus \{1\}),
\\
\mathcal{G}_{19} &= e_{1,1}^5 - e_{1,1}.
\end{align*}
\end{lemma}

\begin{proof}
To obtain the self-reduced set $\mathfrak{G}$, we need to eliminate from $G \cup \bigcup^4_{i=1} G_{i}$ all occurrences of any element of
$\left\{ \mathrm{LM}(u): u \in G \cup \bigcup^4_{i=1} G_{i}\right\}$
as a subword of any element of $G\cup \bigcup^4_{i=1} G_{i}$.  We first note that any element $g \in \bigcup^4_{i=1} G_{i}$ is in normal form modulo
$G \cup \bigcup^4_{i=1} G_{i}\setminus \{g\}$.  So we only consider elements of $G$ (see Definition \ref{1}). For all $  k < i$, we have
\begin{align*}
&\mathcal{R}_1^{(i,j,k,t)} = e_{i,j} e_{j,k} e_{k,t} - e_{k,t} e_{j,k} e_{i,j}- e_{i,t}
\equiv  e_{i,1} e_{1,k} e_{k,t} - e_{k,t} e_{1,k} e_{i,1}- e_{i,t} \!\!\!\!\mod{G_{1}}
\\
& \equiv \delta_{t,1} \left( e_{i,1} e_{1,k} e_{k,1}- e_{1,1}^2 e_{i,1}\right)+ \widehat{\delta}_{t,1}\left( e_{i,1} e_{1,1} e_{1,t} - e_{1,t} e_{1,1} e_{i,1}\right)- e_{i,t}
 \!\!\!\!\mod{G_{3}\cup G_1},
\end{align*}
using the relations $\mathcal{G}^{(i,j,k)}_{1}$, $\mathcal{G}^{(j,k,i)}_{2}$, $\mathcal{G}^{(k,i)}_{6}$, $\mathcal{G}^{(1,k,t)}_{1}$ and $\mathcal{G}^{(k,t,1)}_{2}$. For $t\neq 1$ the last result coincides with $\mathcal{G}_{0}^{(i,t)}$. For $ t = 1$,
we combine the result with the set $\left\{\mathcal{G}_{5}^{(i,k)} |\, \text{for all}\, 1< i<k\right\}\subset G_3$ and obtain the set $\left\{\mathcal{G}_{5}^{(i,k)}|\, \text{for all}\, k\neq i\neq 1\right\}$. For all $ t< j$, we have
\begin{align*}
&\mathcal{R}_2^{(i,j,t)} = e_{i,j} e_{j,i} e_{i,t} - e_{i,t} e_{j,i} e_{i,j}- e_{i,t}
\equiv e_{i,j} e_{j,1} e_{1,t} - e_{1,t} e_{j,1} e_{i,j}- e_{i,t}\!\!\!\!\mod{G_{1}}
\\
& \equiv \delta_{i,1}\left( e_{1,j} e_{j,1} e_{1,t} - e_{1,t} e_{j,1} e_{1,j}\right) + \widehat{\delta}_{i,1}\left(e_{i,1} e_{1,1} e_{1,t} - e_{1,t} e_{1,1} e_{i,1}\right)- e_{i,t}
 \!\!\!\!\mod{G_{1}},
\end{align*}
using the relations $\mathcal{G}^{(j,i,t)}_{1}$, $\mathcal{G}^{(i,t,j)}_{2}$, $\mathcal{G}^{(i,j,1)}_{1}$ and $\mathcal{G}^{(j,1,i)}_{2}$.
For $i \neq 1$ the last result coincides with $\mathcal{G}_{0}^{(i,t)}$(if $t\neq 1$) and $\mathcal{G}_{5}^{(i,1)}$
(if $t =1$). For $i=1$, using the relations $\mathcal{G}^{(i)}_{14}$ and $\mathcal{G}^{(t,j)}_{8}$, we have
\begin{align*}
& \mathcal{R}_2^{(1,j,t)}
 \equiv  \left[\delta_{t,1} \left( e_{1,j} e_{j,1} e_{1,1} - e_{1,1} e_{j,1} e_{1,j}\right)+ \widehat{\delta}_{t,1}\left( e_{1,j} e_{j,1} e_{1,t} - e_{1,t} e_{j,1} e_{1,j}\right)- e_{1,t}\right]
\\ & \!\!\!\!\mod{G_{1}}
\equiv  \delta_{t,1}  \left( e_{1,j} e_{j,1} e_{1,1} - \tfrac{1}{2} e_{1,1}^3 -\tfrac{1}{2} e_{1,1}\right)     +\widehat{\delta}_{t,1}\left( e_{1,j} e_{j,1} e_{1,t} - e_{1,1}^2 e_{1,t}\right)\!\!\!\!\mod{G_{3}}.
\end{align*}
Clearly, for $t = 1$ the normal form of  $\mathcal{R}_2^{(1,j,t)}$ coincides with $\mathcal{G}^{(i)}_{15}$.
For $t\neq 1$, we combine the last result with the set  $\left\{ \mathcal{G}^{(j,t)}_7|\,\, \text{for all}\,\, 1<j<t \right\} \subset G_3$ and obtain the set  $\left\{ \mathcal{G}^{(j,t)}_7|\,\, \text{for all}\,\, 1\neq j\neq t \neq 1\right\}$.
For all $t<i$, we have
\begin{align*}
& \mathcal{R}_3^{(i,j,k,t)} = e_{i,j} e_{k,i} e_{t,k} - e_{t,k} e_{k,i} e_{i,j}+ e_{t,j}
\equiv e_{i,j} e_{1,i} e_{t,1} - e_{t,1} e_{1,i} e_{i,j}+ e_{t,j}\!\!\!\!\mod{G_{1}}
\\
&\equiv \delta_{j,1} (e_{i,1} e_{1,i} e_{t,1} - e_{t,1} e_{1,i} e_{i,1})+ \widehat{\delta}_{j,1}
( e_{1,j} e_{1,1} e_{t,1} - e_{t,1} e_{1,1} e_{1,j})+ e_{t,j}
 \!\!\!\!\mod{G_{1}},
\end{align*}
using the relations $\mathcal{G}^{(k,i,t)}_{2}$, $\mathcal{G}^{(t,k,i)}_{1}$, $\mathcal{G}^{(i,j,1)}_{2}$ and $\mathcal{G}^{(1,i,j)}_{1}$.
For $j\neq 1$ the monic form of the last result coincides with $\mathcal{G}^{(j,t)}_{0}$ (if $t\neq 1$) and $\mathcal{G}^{(j,1)}_{8}$(if $t =1$). For $j=1$, using the relations $\mathcal{G}^{(i)}_{13}$ and $\mathcal{G}^{(t,i)}_{5}$, we have
\begin{align*}
&\mathcal{R}^{(i,1,k,t)}_3
\equiv \delta_{t,1}\left(e_{i,1} e_{1,i} e_{1,1} - \tfrac{1}{2} e_{1,1}^3+\tfrac{1}{2} e_{1,1}        \right) +\widehat{\delta}_{t,1}(e_{i,1} e_{1,i} e_{t,1} - e^2_{1,1} e_{t,1})
 \!\!\!\!\mod{G_{3}}.
\end{align*}
Clearly, for $t =1$ the normal form of $\mathcal{R}^{(i,1,k,t)}_3$ coincides with $\mathcal{G}^{(i)}_{16}$. For $t\neq 1$, we combine the last result
with the set $\left\{\mathcal{G}^{(i,t)}_6| \,\, \text{for all}\,\, 1<i<t   \right\}\subset G_3 $ and obtain the set $\left\{\mathcal{G}^{(i,t)}_6| \,\, \text{for all}\,\, 1\neq i\neq t \neq 1  \right\}$.
For $k<j$, we have
\begin{align*}
&\mathcal{R}_4^{(i, j,k)} = e_{i,j} e_{k,i} e_{i,k} - e_{i,k} e_{k,i} e_{i,j}+ e_{i,j}
 \equiv e_{1,j} e_{k,1} e_{i,k} - e_{i,k} e_{k,1} e_{1,j}+ e_{i,j}\!\!\!\!\mod{G_{1}}
\\
&\equiv \delta_{i,1}(e_{1,j} e_{k,1} e_{1,k} - e_{1,k} e_{k,1} e_{1,j})
  + \widehat{\delta}_{i,1} \left(e_{1,j} e_{1,1} e_{i,1} - e_{i,1} e_{1,1} e_{1,j}\right) + e_{i,j}
\!\!\!\!\mod{G_{1}}
\\
& \equiv \delta_{i,1} \left( e_{1,j} e_{k,1} e_{1,k} - e_{1,1}^2 e_{1,j}+ e_{1,j}\right)+\widehat{\delta}_{i,1} \left(e_{1,j} e_{1,1} e_{i,1} - e_{i,1} e_{1,1} e_{1,j}+e_{i,j}\right) \!\!\!\!\mod{G_{3}},
\end{align*}
using the relations $\mathcal{G}^{(i,j,k)}_{2}\!\! $, $\mathcal{G}^{(k,i,j)}_{1}$\!\!, $\mathcal{G}^{(k,1,i)}_{2}$\!\!, $\mathcal{G}^{(i,k,1)}_{1}$ and $\mathcal{G}^{(k,j)}_{7}$.
For $i \neq 1$ the monic form of the last result coincides with $\mathcal{G}^{(j,i)}_{0}$. For $i=1$, we combine the last result with the set $\left\{ \mathcal{G}^{(j,k)}_{8}| \,\, \text{for all}\,\, 1<j<k\right\}\subset G_3$ and obtain the set
$\left\{\mathcal{G}^{(j,k)}_{8}| \,\, \text{for all}\,\,1\neq j \neq k\right\}$.
For $r< i$, $j\neq k \,\,\text{or}\,\,{t \neq r}$, and $ s\neq k\,\, \text{or}\,\, t \neq i$, we have
\begin{align*}
&\mathcal{R}_5^{(i,j,k,t,r,s)}= e_{i,j}e_{k,t}e_{r,s} - e_{r,s} e_{k,t} e_{i,j}
\equiv \delta_{j,k} \left(e_{i,j} e_{j,t} e_{r,s} - e_{r,s} e_{j,t} e_{i,j}\right)\\& + \widehat{\delta}_{j,k}
\delta_{i,t} \left( e_{i,j}e_{k,i}e_{r,s} - e_{r,s} e_{k,i} e_{i,j}\right)\!\!\!\!\mod{G_{1}}
\equiv
\delta_{j,k}\Big[ \delta_{i,t} \left(e_{i,j} e_{j,i} e_{r,s} - e_{r,s} e_{j,i} e_{i,j}\right) \\&+ \widehat{\delta}_{i,t}\left(e_{i,1} e_{1,t} e_{r,s} - e_{r,s} e_{1,t} e_{i,1}\right) \Big]    \!\!\!\!\mod{G_{1}}\equiv
 \delta_{j,k}\left[ \delta_{s,1} \widehat{\delta}_{i,t}\left(e_{i,1} e_{1,t} e_{r,1} - e_{r,1} e_{1,t} e_{i,1}\right) \right]\\& \!\!\!\!\mod{G_{1}}\equiv 0 \!\!\!\!\mod{G_{3}},
\end{align*}
using the relations $\mathcal{G}^{(i,j,k,t)}_{3}\!\!$, $\mathcal{G}^{(k,t,i,j)}_{3}\!\!$, $\mathcal{G}^{(i,j,t)}_{1}\!\!$, $\mathcal{G}^{(j,t,i)}_{2}\!\!$, $\mathcal{G}^{(j, i, r,s)}_{3}\!\!$, $\mathcal{G}^{(r,s,j,i)}_{3}\!\!$, $\mathcal{G}^{(1,t,r,s)}_{3}\!\!$, $\mathcal{G}^{(r,s,1,t)}_{3}\!\!$,  $\mathcal{G}^{(i,t,r)}_{11}\!\!$ and  $\mathcal{G}^{(r,t,i)}_{11}\!\!$.
Similarly, we can show that  $\mathcal{R}_6^{(i,j,k,t,s)}\equiv  0 \!\!\mod{G_1 \cup G_{3}}$.
\end{proof}
The following lemma plays a crucial role in proving that the set $\mathfrak{G}$ of Lemma \ref{G*} is a Gr\"obner basis for the ideal $I$.
\begin{lemma}\label{Gh}
For the universal enveloping algebra $\mathfrak{A}$, either
\[
 (i)\, \dim \mathfrak{A}= \infty, \text{or}\qquad
(ii)\, \dim \mathfrak{A}< \infty\,\, \text{and} \,\, \dim \mathfrak{A}\geq 4n^2 +1.\]
\end{lemma}

\begin{proof}
Suppose that $\dim(\mathfrak{A})< \infty$. We show that over an algebraically closed field $F$, there exist four inequivalent irreducible representations of degree $n$ of the anti-Jordan triple system $\mathfrak{J}$, in addition to the trivial representation of degree $1$. For $k = 1,\dots,4$, we define the following maps:
\begin{align*}
\rho_k& : \mathfrak{J} \to \mathrm{End}\, V_k,
\\
&\rho_1 (E_{i,j}) = E_{i,j},\quad \rho_2 (E_{i,j})= -E_{i,j},\quad \rho_3(E_{i,j})= \mathrm{I}\, E_{j,i},\quad \rho_4(E_{i,j}) = -\mathrm{I}\, E_{j,i},
\end{align*}
where $\mathrm{I} = \sqrt{-1}$. Our first step is to show that the maps $\rho_{k}$, $k= 1,\dots,4$ are representations
of the anti-Jordan triple system $\mathfrak{J}$. Clearly $\rho_1$ is a representation (the natural representation). For $\rho_2$, we have
\begin{align*}
\rho_2\left(\langle E_{i,j}, E_{k,\ell}, E_{r,s}\rangle\right) &=  \rho_2\left( \delta_{j,k} \delta_{\ell,r} E_{i,s} - \delta_{s,k}\delta_{\ell,i} E_{r,j}\right)
 = -\delta_{j,k} \delta_{\ell,r} E_{i,s} + \delta_{s,k}\delta_{\ell,i} E_{r,j},
\end{align*}
on the other hand, we have
\begin{align*}
\langle \rho_2(E_{i,j}), \rho_2(E_{k,\ell}), \rho_2(E_{r,s})    \rangle & =
 \rho_2 (E_{i,j}) \rho_2( E_{k,\ell})\rho_2(E_{r,s}) - \rho_2(E_{r,s}) \rho_2(E_{k,\ell}) \rho_{2}(E_{i,j})
\\
&= -\delta_{j,k} \delta_{\ell,r} E_{i,s} + \delta_{s,k} \delta_{\ell,i} E_{r,j}.
\end{align*}
Thus $ \rho_2$ is a representation. For $\rho_3$, we have
\begin{align*}
\rho_3\left(\langle E_{i,j}, E_{k,\ell}, E_{r,s}\rangle\right) &= \rho_3\left( \delta_{j,k} \delta_{\ell,r} E_{i,s} - \delta_{s,k}\delta_{\ell,i} E_{r,j}\right)= \delta_{j,k} \delta_{\ell,r}\,\mathrm{I} E_{s,i} - \delta_{s,k}\delta_{\ell,i}\, \mathrm{I} E_{j,r},
\end{align*}
on the other hand, we have
\begin{align*}
\langle \rho_3(E_{i,j}), \rho_3(E_{k,\ell}), \rho_3(E_{r,s})    \rangle & =
 \rho_3 (E_{i,j}) \rho_3( E_{k,\ell})\rho_3(E_{r,s}) - \rho_3(E_{r,s}) \rho_3(E_{k,\ell}) \rho_{3}(E_{i,j})
\\
&= -\, \delta_{i,\ell} \delta_{k,s}\, \mathrm{I} E_{j,r} + \delta_{r,\ell} \delta_{k,j}\, \mathrm{I} E_{s,i}.
\end{align*}
Similarly, we can show that $\rho_4$ is a representation. We now show that for all $i,j = 1,\dots, 4 $ and $i\neq j$, the representations $\rho_{i}$ and $\rho_{j}$ are inequivalent. Indeed, there is no matrix $T$ so that
 \[ \rho_i(x) = T^{-1} \rho_j(x)\, T, \quad \text{for all} \,\, x \in  \mathfrak{J},\, i\neq j.\]
This is easily seen by checking the trace on the both sides and using the definitions of the representations
: $\mathrm{Tr} (\rho_i(x)) \neq \mathrm{Tr}(  T^{-1} \rho_j(x)\, T) = \mathrm{Tr}(\rho_{j}(x))$.
The representations $\rho_{i}$, $i= 1,\dots,4$ of $\mathfrak{J}$ can be extended to representations of the universal envelope $\mathfrak{A}$.  Hence $\mathfrak{A}$ has a subalgebra of dimension $4n^2+1$, which is isomorphic to the direct sum of the matrix algebras corresponding to these representations.
\end{proof}
We now can state the main theorem of this section.
\begin{theorem}\label{ba}
With notation as above. If $\mathfrak{J}$ is the anti-Jordan triple system of all $n \times n$ matrices $( n\geq 2)$ then:
 \begin{enumerate}[$(i)$]
   \item  The set $\mathfrak{G}$ is a Gr\"obner basis for the ideal $I$. \label{poo}
   \item The universal enveloping algebra $\mathfrak{A}$ of $\mathfrak{J}$ is finite-dimensional with basis $\mathfrak{B}$ consists of $ 4 n^2+1$ monomials:\label{po}
\begin{align*}
 \mathfrak{B} & = \left\{1,\, e_{i,j},\,   e_{i,1} e_{1,j},\, e_{1,1}^2 e_{1,j},\, e^4_{1,1}|\, i,j\in \Omega  \right\}\cup   \{ e_{1,i} e_{j,1}|\, i,j \in \Omega, (i,j)\neq (1,1) \}
\\& \quad\cup \{ e_{1,i} e_{1,1} e_{j,1}|\, i,j \in \Omega,\,  j \neq 1\}.     \end{align*}
 \end{enumerate}
\end{theorem}
\begin{proof}
By Lemma \ref{G*} the set $\mathfrak{G}$ is the self-reduced form of the set $ G \cup \bigcup^4_{i=1}G_{i}$, so it remains to show that $\mathfrak{G}$ is closed under any composition.
We note first that there are $4n^2 +1$ monomials of $F\langle X \rangle$ that do not have leading monomials of $\mathfrak{G}$ as factors, namely,
\[
1, \quad
e_{i,j}, \quad
e_{i,1} e_{1,j}, \quad
e_{1,i} e_{j,1}, \, (i,j)\neq (1,1),  \quad
e_{1,i} e_{1,1} e_{j,1}, \, j \neq 1, \quad
e_{1,1}^2 e_{1,j}, \quad
e^4_{1,1},
\]
{for all} $ i,j \in \Omega $. Suppose on the contrary that $\mathfrak{G}$ is not a Gr\"obner basis for the ideal $I$. Then $\mathfrak{G}$ is not closed under at least one composition by Theorem \ref{di}, i.e., there exist $f, g \in \mathfrak{G}$ such that $fx - y g
\not\equiv 0\!\!\! \mod \mathfrak{G}$. We add the normal form of $fx - yg$ to the set $\mathfrak{G}$. Hence, the number of the monomials of $F\langle X \rangle$ that do not have the leading monomials of $\mathfrak{G}$ as factors is less than $4n^2 +1$. Hence, $\dim \mathfrak{A} < 4n^2 +1$. But Lemma \ref{Gh} implies that $\dim \mathfrak{A} \geq 4n^2+1$, which is a contradiction. This shows that $\mathfrak{G}$ is a Gr\"obner basis for the ideal $I$. The proof \eqref{po} is obvious by using \eqref{poo} and Proposition \ref{C(I)proposition}.
\end{proof}


\section{The structure constants of $\mathfrak{A}$}\label{stc}

In this section we use Theorem \ref{ba} and the relations of Lemma \ref{G*} to compute the structure constants of the universal enveloping algebra $\mathfrak{A}$.

\begin{lemma}\label{auto}
Define an anti-automorphism $\eta\colon F \langle X  \rangle \to F \langle X  \rangle$ of the free associative algebra
generated by $X=\{ e_{i,j}\}_{{i,j}\in \Omega} $  by $\eta(e_{i,j}) = e_{j,i}$. Then
  $\eta$ induces an anti-automorphism of order $2$ on $\mathfrak{A}$ (also denoted $\eta$).
  \end{lemma}
\begin{proof}
It suffices to show that the ideal $I= \langle \mathfrak{G} \rangle$ (see Theorem \ref{ba}) is invariant under the action of $\eta$. We have, for example,
\begin{align*}
&\eta \left( \mathcal{G}^{(i,j)}_{0}\right) =  e_{j,1} e_{1,1} e_{1,i} - e_{1,i} e_{1,1} e_{j,1} - e_{j,i}
 =  \mathcal{G}^{(j,i)}_{0},
\\
&\eta \left( \mathcal{G}^{(i,j,k)}_{1}\right)  = e_{k,j}e_{j,i} - e_{k,1} e_{1,i} = \mathcal{G}_1^{(k,j,i)}.
\end{align*}
A similar argument applies to all the other elements of $\mathfrak{G}$.
\end{proof}
The next seven Propositions give the explicit structure constants of $\mathfrak{A}$.
\begin{proposition}\label{sf}
Let $i,j,k,\ell \in \Omega$. Then in $\mathfrak{A}$, we have

\begin{dgroup*} 
\begin{dmath}
\label{l1}
e_{i,j}\cdot e_{k,\ell} =
\delta_{j,k} \left[
 \delta_{i,\ell}\left\{\left(\delta_{i,1} + \widehat{\delta}_{i,1}\widehat{\delta}_{j,1}\right) e_{1,j} e_{j,1}
+\left( \delta_{j,1}\widehat{\delta}_{i,1}+\widehat{\delta}_{i,1}\widehat{\delta}_{j,1} \right)e_{i,1} e_{1,i}
  - \widehat{\delta}_{i,1}\widehat{\delta}_{j,1} e_{1,1}^2\right\}
+ \widehat{\delta}_{i,\ell} e_{i,1}e_{1,\ell}\right]
+ \widehat{\delta}_{j,k}\delta_{i,\ell}e_{1,j}e_{k,1},
\end{dmath}
\begin{dmath}
\label{l2}
e_{i,j}\cdot e_{k,1}e_{1,\ell}= \delta_{i,1} \Big[ \left(\delta_{j,k}\delta_{\ell,1}\left( \delta_{\ell,j}+ \tfrac{1}{2}\widehat{\delta}_{\ell,j} \right)+ \tfrac{1}{2}
\widehat{\delta}_{j,k}\delta_{k,\ell}\delta_{j,1}\right)
 e_{1,1}^3+ \widehat{\delta}_{j,k}\delta_{k,\ell}\widehat{\delta}_{j,1}e_{1,1}^2 e_{1,j}
  + \delta_{j,k}\left( 2 \delta_{\ell,j} \widehat{\delta}_{\ell,1} + \widehat{\delta}_{\ell,j} \left( \delta_{j,1} + \widehat{\delta}_{j,1} \widehat{\delta}_{\ell,1}\right)\right) e_{1,1}^2 e_{1,\ell}
 + {\tfrac{1}{2} \left(\delta_{\ell,1}\delta_{j,k} \widehat{\delta}_{\ell,j}- \widehat{\delta}_{j,k}\delta_{k,\ell} \delta_{j,1}    \right) e_{1,1}
  - \left(\delta_{j,k}\delta_{\ell,j} \widehat{\delta}_{\ell,1} +  \widehat{\delta}_{j,k}\delta_{k,\ell}\widehat{\delta}_{j,1}\right) e_{1,j}\Big]}+\widehat{\delta}_{i,1}\delta_{j,k} ( e_{1,\ell} e_{1,1} e_{i,1} + e_{i,\ell}),
\end{dmath}
\begin{dmath}
\label{l22}
e_{\ell,1} e_{1,k} \cdot e_{j,i} =
\delta_{i,1} \Big[ \left(\delta_{j,k}\delta_{\ell,1}\left( \delta_{\ell,j}+ \tfrac{1}{2}\widehat{\delta}_{\ell,j} \right)+ \tfrac{1}{2}
\widehat{\delta}_{j,k}\delta_{k,\ell}\delta_{j,1}\right)
 e_{1,1}^3+\widehat{\delta}_{j,k}\delta_{k,\ell}\widehat{\delta}_{j,1} e_{1,1}^2 e_{j,1}
     + \delta_{j,k}\left( 2 \delta_{\ell,j} \widehat{\delta}_{\ell,1} + \widehat{\delta}_{\ell,j} \left( \delta_{j,1} + \widehat{\delta}_{j,1} \widehat{\delta}_{\ell,1}\right)\right) \left(e_{1,1}^2 e_{\ell,1}+e_{\ell,1}\right)
 + {\tfrac{1}{2} \Big(\delta_{\ell,1}\delta_{j,k} \widehat{\delta}_{\ell,j} - \widehat{\delta}_{j,k}\delta_{k,\ell} \delta_{j,1}    \Big) e_{1,1}
 - \delta_{j,k}\delta_{\ell,j} \widehat{\delta}_{\ell,1} e_{j,1}\Big]}+\widehat{\delta}_{i,1}\delta_{j,k}
  \left( \delta_{\ell,1} e_{1,1}^2 e_{1,i}+ \widehat{\delta}_{\ell,1}\left( e_{1,i} e_{1,1} e_{\ell,1} + e_{\ell,i}\right)\right),
\end{dmath}
\begin{dmath}
\label{l3}
e_{i,j}\cdot e_{1,k} e_{\ell,1} = \delta_{j,1}\Big[ \tfrac{1}{2}\left( \delta_{i,k} \delta_{\ell,1}\widehat{\delta}_{i,1}
+\widehat{\delta}_{i,k}\widehat{\delta}_{\ell,1}\delta_{k,\ell}{\delta}_{i,1}\right)e_{1,1}^3+\widehat{\delta}_{i,k}\delta_{k,\ell}\widehat{\delta}_{i,1}\widehat{\delta}_{\ell,1} \, e_{1,1}^2 e_{i,1}
+  \delta_{i,k} \left( \delta_{i,1}
 + \widehat{\delta}_{i,1}\widehat{\delta}_{\ell ,1} \left(2\delta_{i,\ell}
+ \widehat{\delta}_{i,\ell}\right)\right)
e_{1,1}^2 e_{\ell,1}
+\tfrac{1}{2}\left(\widehat{\delta}_{i,k} \widehat{\delta}_{\ell,1}\delta_{k,\ell}\delta_{i,1} -\delta_{i,k} \delta_{\ell,1}\widehat{\delta}_{i,1}
\right)e_{1,1}
 +\widehat{\delta}_{i,1}\widehat{\delta}_{\ell,1}
 \left(\delta_{i,k} \delta_{i,\ell}
+\widehat{\delta}_{i,k}\delta_{k,\ell}  \right)e_{i,1}\Big]
+\widehat{\delta}_{j,1} \delta_{i,k} \left( \delta_{\ell,1} \left(e_{1,1}^2 e_{1,j} - e_{1,j}\right)
+\widehat{\delta}_{\ell,1} e_{1,j}e_{1,1}e_{\ell,1}\right);
\quad {(k,\ell)\neq (1,1)},
\end{dmath}
\begin{dmath}
\label{l4}
e_{1,\ell} e_{k,1}\cdot e_{j,i}
 = \delta_{j,1}\Big[ \tfrac{1}{2}\left( \delta_{i,k} \delta_{\ell,1}\widehat{\delta}_{i,1}
+\widehat{\delta}_{i,k}\widehat{\delta}_{\ell,1}\delta_{k,\ell}{\delta}_{i,1}\right)e_{1,1}^3+\widehat{\delta}_{i,k}\delta_{k,\ell}\widehat{\delta}_{i,1}\widehat{\delta}_{\ell,1}\, e_{1,1}^2 e_{1,i}
+  \delta_{i,k} \left( \delta_{i,1}
 + \widehat{\delta}_{i,1}\widehat{\delta}_{\ell ,1} \left(2\delta_{i,\ell}
+
\widehat{\delta}_{i,\ell}\right)\right)
\left(e_{1,1}^2  e_{1,\ell} -e_{1,\ell}  \right)
 +\tfrac{1}{2}\left(
\widehat{\delta}_{i,k}\widehat{\delta}_{\ell,1}\delta_{k,\ell}\delta_{i,1}  -\delta_{i,k} \delta_{\ell,1}\widehat{\delta}_{i,1}\right)e_{1,1}
 +\widehat{\delta}_{i,1}\widehat{\delta}_{\ell,1}
 \delta_{i,k} \delta_{i,\ell}
 e_{1,i}\Big]
+\widehat{\delta}_{j,1} \delta_{i,k} \left( \delta_{\ell,1}  e_{1,1}^2 e_{j,1}
+\widehat{\delta}_{\ell,1} e_{1,\ell}e_{1,1}e_{1,j}\right);\qquad  {(k,\ell)\neq (1,1)}.
\end{dmath}
\end{dgroup*}

\end{proposition}

\begin{proof}
For \eqref{l1}, we use the relations $\mathcal{G}^{(i,j,k,\ell)}_{3}$, $\mathcal{G}^{(i,j,\ell)}_{1}$ and $\mathcal{G}^{(i,j,k)}_{2}$ and get
\begin{align*}
e_{i,j}\cdot e_{k,\ell} & =\delta_{j,k} \left(\delta_{i,\ell} e_{i,j} e_{j,i} + \widehat{\delta}_{i,\ell} e_{i,1}e_{1,\ell}\right)+ \widehat{\delta}_{j,k}\delta_{i,\ell}e_{1,j}e_{k,1}.
\end{align*}
Using the relation $\mathcal{G}^{(i,j)}_{4}$ implies
\begin{dmath*}
 e_{i,j}\cdot e_{k,\ell} = \delta_{j,k} \Big[\delta_{i,\ell}\left( \delta_{i,1} e_{1,j} e_{j,1}
+ \delta_{j,1}\widehat{\delta}_{i,1}e_{i,1} e_{1,i}
 + \widehat{\delta}_{i,1}\widehat{\delta}_{j,1}\left(e_{i,1} e_{1,i} +e_{1,j} e_{j,1} - e_{1,1}^2  \right)\right)
+ \widehat{\delta}_{i,\ell} e_{i,1}e_{1,\ell}\Big]
+ \widehat{\delta}_{j,k}\delta_{i,\ell}e_{1,j}e_{k,1}.
\end{dmath*}
This completes the proof of \eqref{l1}.
For\eqref{l2}, we use \eqref{l1} (of the present proposition) and obtain
\begin{align*}
\left(e_{i,j} e_{k,1}\right)e_{1,\ell}
 &= \delta_{j,k} \left(\delta_{i,1} e_{1,j} e_{j,1}
 e_{1,\ell}
+ \widehat{\delta}_{i,1} e_{i,1}e_{1,1}e_{1,\ell}\right)
+ \widehat{\delta}_{j,k}\delta_{i,1}e_{1,j}e_{k,1}e_{1,\ell}
\\
& = \delta_{i,1} \left( \delta_{j,k} e_{1,j} e_{j,1} e_{1,\ell}+ \widehat{\delta}_{j,k} e_{1,j} e_{k,1} e_{1,\ell}\right) + \widehat{\delta}_{i,1}\delta_{j,k} e_{i,1} e_{1,1} e_{1,\ell}.
\end{align*}
We now write  \[ A = e_{1,j} e_{j,1}  e_{1,\ell},\qquad B = \widehat{\delta}_{j,k} e_{1,j} e_{k,1} e_{1,\ell},\]
and use the relations $\mathcal{G}_{0}^{(i,\ell)}$ ( if $\ell \neq 1$) and $\mathcal{G}^{(i,1)}_{5}$ (if $\ell=1$) for the last term to obtain
\begin{align}\label{20}
\left(e_{i,j} e_{k,1}\right)e_{1,\ell} & = \delta_{i,1} \left( \delta_{j,k}\, A
+ B\right)
+ \widehat{\delta}_{i,1}\delta_{j,k} ( e_{1,\ell} e_{1,1} e_{i,1} + e_{i,\ell}).
\end{align}
Using the relations $\mathcal{G}_{10}^{(j)}$, $\mathcal{G}_{15}^{(j)}$ and $\mathcal{G}_{7}^{(j,\ell)}$ gives
\begin{align*}
A
&= \delta_{\ell,j}  \left[ \delta_{\ell,1} e_{1,1}^3 + \widehat{\delta}_{\ell,1}\left(2 e_{1,1}^2 e_{1,j} - e_{1,j}\right)\right]
  \\&\quad + \widehat{\delta}_{\ell,j}\left[ \delta_{j,1} e_{1,1}^2 e_{1,\ell}
 + \widehat{\delta}_{j,1}\left(\delta_{\ell,1} \tfrac{1}{2} \left( e_{1,1}^3 + e_{1,1}\right) + \widehat{\delta}_{\ell,1}e_{1,1}^2 e_{1,\ell}\right)\right]
 \\
 & = \delta_{\ell,1}\left[\left( \delta_{\ell,j}+ \tfrac{1}{2}\widehat{\delta}_{\ell,j} \right)  e_{1,1}^3 + \tfrac{1}{2} \widehat{\delta}_{\ell,j} e_{1,1}\right]
 + \left[ 2 \delta_{\ell,j} \widehat{\delta}_{\ell,1} + \widehat{\delta}_{\ell,j} ( \delta_{j,1} + \widehat{\delta}_{j,1} \widehat{\delta}_{\ell,1})\right] e_{1,1}^2 e_{1,\ell}\\&\quad - \delta_{\ell,j} \widehat{\delta}_{\ell,1} e_{1,j}.
\end{align*}
Using the relations $\mathcal{G}^{(j,k,\ell)}_{12}$,
 $\mathcal{G}^{(\ell)}_{14}$ and $\mathcal{G}^{(j,\ell)}_8$ gives
\begin{align*}
B &= \widehat{\delta}_{j,k} \delta_{k,\ell}  e_{1,j} e_{\ell,1} e_{1,\ell}
=  \widehat{\delta}_{j,k}\delta_{k,\ell} \left( \delta_{j,1} \tfrac{1}{2} \left(e_{1,1}^3 - e_{1,1}\right)
 + \widehat{\delta}_{j,1} (e_{1,1}^2 e_{1,j} - e_{1,j})\right).
\end{align*}
Using $A$ and $B$ in \eqref{20} and combining the coefficients completes the proof of \eqref{l2}.
The proof of \eqref{l22} is obvious by applying the anti-automorphism $\eta$ (see Lemma \ref{auto}) to both sides of \eqref{l2} (of the present Proposition) and using the relations $\mathcal{G}_{5}^{(j,1)}$, $\mathcal{G}_{5}^{(\ell,1)}$ and $\mathcal{G}_{8}^{(i,1)}$.
The proofs of \eqref{l3} and \eqref{l4} are similar.
\end{proof}
\begin{proposition}\label{ss}
Let $i,j,k, \ell \in \Omega$. Then in $\mathfrak{A}$, we have
\begin{dgroup*} 
\begin{dmath}
 e_{i,j} \cdot e_{1,k} e_{1,1} e_{\ell,1} = {\delta_{i,k}\left[- e_{1,j} e_{\ell,1}
+ \widehat{\delta}_{j,1} \delta_{j,\ell} \tfrac{1}{2} \left(e_{1,1}^4 + e_{1,1}^2\right)
 \right]}; \quad  {\ell \neq 1}, \label{g1}
\end{dmath}
\begin{dmath}
\label{g2}
e_{i,j} \cdot e_{1,1}^2 e_{1,k}
= \delta_{j,1}\left[ \delta_{i,1} \left(\delta_{k,1} e_{1,1}^4 + \widehat{\delta}_{k,1} e_{1,1}e_{1,k}\right)
+\widehat{\delta}_{i,1} \left( \delta_{i,k} \tfrac{1}{2} (e_{1,1}^4 - e_{1,1}^2) + e_{i,1} e_{1,k}\right)\right]
- \widehat{\delta}_{j,1} \delta_{i,1}\delta_{k,1} e_{1,j} e_{1,1},\notag
\end{dmath}
\begin{dmath}
e_{i,j} \cdot e_{1,1}^4 =  \delta_{j,1}\left[\delta_{i,1} e_{1,1} + \widehat{\delta}_{i,1}\left( e_{1,1}^2 e_{i,1} + e_{i,1}\right)\right]
- \widehat{\delta}_{j,1} \delta_{i,1} \left(e_{1,1}^2 e_{1,j} -e_{1,j}\right)
\label{g3}.
\end{dmath}
\end{dgroup*}
\end{proposition}

\begin{proof}
For \eqref{g1}, let $\ell \neq 1$ and consider two cases.
Case I. If $k=1$ then \eqref{l2} of Proposition \ref{sf} implies
\begin{align}\label{l11}
e_{i,j} e_{1,1}^2 = \delta_{j,1} \left[ \delta_{i,1} e_{1,1}^3 + \widehat{\delta}_{i,1}(e_{1,1}^2 e_{i,1}+e_{i,1})\right]
& +  \widehat{\delta}_{j,1}\delta_{i,1} \left(e_{1,1}^2 e_{1,j} - e_{1,j}\right).
\end{align}
Multiply \eqref{l11} by $e_{\ell,1}$ and use the relations $\mathcal{G}^{(\ell)}_{18}$, $\mathcal{G}^{(i,1,\ell,1)}_{3}$, $\mathcal{G}^{(1,j,\ell)}_{11}$ and $\mathcal{G}^{(j)}_{13}$ and obtain
\begin{dmath*}
\left(e_{i,j} e_{1,1}^2\right) e_{\ell,1}=  - \delta_{j,1}\delta_{i,1} e_{1,1}e_{\ell,1}
 +  \widehat{\delta}_{j,1}\delta_{i,1} \left( \delta_{j,\ell} \tfrac{1}{2} \left(e_{1,1}^4 + e_{1,1}^2\right)- e_{1,j}e_{\ell,1}\right).
\end{dmath*}
Case II. If $k\neq 1$ then \eqref{l3} of Proposition \ref{sf} implies
\begin{dmath*}
 \left(e_{i,j} e_{1,k} e_{1,1}\right) e_{\ell ,1}
 = \delta_{j,1}\left[ \delta_{i,k} \left( \delta_{i,1} e_{1,1}^3 e_{\ell,1}
 + \widehat{\delta}_{i,1}\tfrac{1}{2}\left(e_{1,1}^3-e_{1,1}\right) e_{\ell,1}\right) \right]
 + \widehat{\delta}_{j,1}\delta_{i,k}  \left(e_{1,1}^2 e_{1,j} - e_{1,j}\right) e_{\ell,1}.
\end{dmath*}
Using the relations $\mathcal{G}^{(\ell)}_{18}$, $\mathcal{G}^{(1,j,\ell)}_{11}$ and $\mathcal{G}^{(j)}_{13}$ gives
\begin{dmath*}
\left(e_{i,j} e_{1,k} e_{1,1}\right) e_{\ell ,1}
 = \delta_{j,1}\left[ \delta_{i,k} \left( -\delta_{i,1} e_{1,1} e_{\ell,1}
 + \widehat{\delta}_{i,1}\tfrac{1}{2}(-e_{1,1}-e_{1,1}) e_{\ell,1}\right) \right]
 + \widehat{\delta}_{j,1}\delta_{i,k} \left[\delta_{j,\ell} \tfrac{1}{2}\left(e_{1,1}^4+e_{1,1}^2\right)-  e_{1,j} e_{\ell,1}\right]
=  -\delta_{j,1}\delta_{i,k} e_{1,1} e_{\ell,1}
 + \widehat{\delta}_{j,1}\delta_{i,k} \left[\delta_{j,\ell} \tfrac{1}{2}\left(e_{1,1}^4+e_{1,1}^2\right)-  e_{1,j} e_{\ell,1}\right].
\end{dmath*}
Combining the results of the two cases completes the proof of \eqref{g1}. For \eqref{g2}, we multiply \eqref{l11} by $e_{1,k}$ and use the relations $\mathcal{G}^{(k)}_{17}$, $\mathcal{G}^{(1,i,k)}_{12}$, $\mathcal{G}^{(i)}_{14}$, $\mathcal{G}^{(1,j,1)}_{11}$ and $\mathcal{G}^{(1,j,1,k)}_{3} $. The proof of \eqref{g3} is similar.
\end{proof}
The proofs of the next five Propositions are similar to the proofs of Propositions \ref{sf} and \ref{ss} and are omitted.
\begin{proposition}\label{st}
Let $i,j,k,\ell \in \Omega$. Then in $\mathfrak{A}$, we have
\begin{dgroup*}
\begin{dmath}
e_{i,1}e_{1,j}\cdot e_{k,1} e_{1,\ell} =  \delta_{j,k}  \delta_{\ell,1}\left(\delta_{\ell,j}  \delta_{i,1} e_{1,1}^4 +  \tfrac{1}{2}   \widehat{\delta}_{\ell,j}    \delta_{i,1} \left(e_{1,1}^4+e_{1,1}^2\right)
\right)
   + \tfrac{1}{2}\left[
 \delta_{j,k}\widehat{\delta}_{i,1} \delta_{i,\ell} \left( 2 \delta_{\ell,j} \widehat{\delta}_{\ell,1} + \widehat{\delta}_{\ell,j}\left(\delta_{j,1}
+ \widehat{\delta}_{j,1}\widehat{\delta}_{\ell,1}\right)\right)+ \widehat{\delta}_{j,k}\delta_{k,\ell}\left(\widehat{\delta}_{j,1}\widehat{\delta}_{i,1} \delta_{i,j}+ \delta_{j,1} \delta_{i,1}\right)\right]
 \left(e_{1,1}^4 - e_{1,1}^2\right) +\delta_{j,k} \left(\delta_{\ell,1} \widehat{\delta}_{i,1}+  \delta_{\ell,j} \widehat{\delta}_{\ell,1} + \widehat{\delta}_{\ell,j}\left(\delta_{j,1}
+ \widehat{\delta}_{j,1}\widehat{\delta}_{\ell,1}\right) \right) e_{i,1} e_{1,\ell},\label{g4}
 \end{dmath}
\begin{dmath}
 e_{i,1} e_{1,j} \cdot e_{1,k} e_{\ell,1} = \delta_{j,1} \widehat{\delta}_{k,1}\delta_{k,\ell}
 \tfrac{1}{2}\left(\delta_{i,1} \left(e_{1,1}^4 + e_{1,1}^2\right)+ 2\widehat{\delta}_{i,1}e_{i,1}e_{1,1}\right)
   -\delta_{k,1} \left(\delta_{j,1}\delta_{i,1} +\widehat{\delta}_{j,1} \widehat{\delta}_{\ell,1} \delta_{i,j}   \right)e_{1,1}e_{\ell,1};
 \quad {(k,\ell)\neq(1,1)},\label{g5}
 \end{dmath}
\begin{dmath}
 e_{i,1} e_{1,j} \cdot e_{1,k} e_{1,1} e_{\ell,1} = -\delta_{k,1}\left(\delta_{i,1}\delta_{j,1}+\widehat{\delta}_{j,1}\delta_{i,j}\right)e_{1,1}^2 e_{\ell,1};\quad {\ell\neq 1},\label{g55}
\end{dmath}
\begin{dmath}
e_{i,1} e_{1,j} \cdot e_{1,1}^2 e_{1,k}=   \delta_{j,1} \left[ \delta_{k,1} \left(\delta_{i,1}e_{1,1}+ \widehat{\delta}_{i,1}\left(e_{1,1}^2e_{i,1}+e_{i,1}\right)\right )
+\widehat{\delta}_{k,1} \left(\delta_{i,1} e_{1,1}^2 e_{1,k}
 +\widehat{\delta}_{i,1}  \left( e_{1,k} e_{1,1}e_{i,1}+e_{i,k}\right) \right)\right] -\widehat{\delta}_{j,1}\delta_{k,1} \delta_{i,j}\tfrac{1}{2} \left(e_{1,1}^3 - e_{1,1} \right),\label{g6}
\end{dmath}
\begin{dmath}
 e_{i,1} e_{1,j} \cdot e_{1,1}^4 =  \delta_{j,1} e_{i,1} e_{1,1} + \widehat{\delta}_{j,1} \delta_{j,i} \tfrac{1}{2} \left(e_{1,1}^2 - e_{1,1}^4 \right).\label{g7}
\end{dmath}
\end{dgroup*}
\end{proposition}
\begin{proposition}\label{LL}
Let $i,j,k,\ell \in \Omega$ and $(i,j)\neq (1,1)$. Then in $\mathfrak{A}$, we have
\begin{dgroup*}
\begin{dmath}\label{k1}
 e_{1,i}e_{j,1}\cdot e_{1,k} e_{\ell,1}
 =\tfrac{1}{2}\left[ \delta_{i,1} \delta_{j,k} \delta_{\ell,1}\widehat{\delta}_{j,1} \left(e_{1,1}^4 - e_{1,1}^2\right)
+\left\{ \delta_{j,k}\widehat{\delta}_{i,1} \delta_{i,\ell} \left( \delta_{j,1}
 + \widehat{\delta}_{j,1}\widehat{\delta}_{\ell ,1} \left(2\delta_{j,\ell}
+
\widehat{\delta}_{j,\ell}\right)\right) + \widehat{\delta}_{j,k}\widehat{\delta}_{\ell,1} \delta_{k,\ell}\left( \widehat{\delta}_{j,1} \widehat{\delta}_{j,\ell} \delta_{i,j}+\delta_{i,1}{\delta}_{j,1} \right)\right\}
\left(e_{1,1}^4 +e_{1,1}^2\right)\right]
-  {\delta_{j,k} \left(\delta_{j,1}+ \widehat{\delta}_{j,1}\left(\delta_{\ell,1}
 \widehat{\delta}_{i,1}
 + \widehat{\delta}_{\ell ,1}\right) \right)
e_{1,i} e_{\ell,1}
;\quad
 {(k,\ell)\neq (1,1)}},
\end{dmath}
\begin{dmath}
\label{k3}
e_{1,i} e_{j,1}  \cdot e_{1,1}^2 e_{1,k} = \left( - \delta_{j,1} \widehat{\delta}_{i,1} \delta_{k,1} +  \widehat{\delta}_{j,1}\delta_{i,j} \delta_{k,i}\right)e_{1,1}^2 e_{1,i}
+ \delta_{j,1}\delta_{k,1} \widehat{\delta}_{i,1}e_{1,i}
+\delta_{i,j} \widehat{\delta}_{j,1}\left( \tfrac{1}{2} \delta_{k,1}  \left(e_{1,1}^3 +e_{1,1}\right)
+  \widehat{\delta}_{k,i}\widehat{\delta}_{k,1}e_{1,1}^2 e_{1,k}\right),
\end{dmath}
\begin{dmath}
\label{k4}
e_{1,i} e_{j,1} \cdot e_{1,1}^4 = \delta_{j,1} e_{1,i} e_{1,1} + \widehat{\delta}_{j,1}  \widehat{\delta}_{i,1}\delta_{i,j}\tfrac{1}{2}  \left(e_{1,1}^4 + e_{1,1}^2 \right),
\end{dmath}
\begin{dmath}
\label{k2}
e_{1,i} e_{j,1} \cdot e_{1,k} e_{1,1} e_{\ell,1} = - \delta_{j,k} e_{1,i} e_{1,1} e_{\ell,1}; \quad{\ell\neq 1}.
\end{dmath}
\end{dgroup*}
\end{proposition}
\begin{proposition}\label{466}
Let $i,j,k\in \Omega$. Then in $\mathfrak{A}$, we have
\begin{dgroup*}
\begin{dmath}
\label{De2}
 e_{1,1}^2 e_{1,k} \cdot e_{i,j}
  = \delta_{k,i} \left[\delta_{j,1} \left( \left(\delta_{k,1} +\tfrac{1}{2}\widehat{\delta}_{k,1}\right) e_{1,1}^4
  + \tfrac{1}{2}\widehat{\delta}_{k,1}  e_{1,1}^2 \right)
+\widehat{\delta}_{j,1}  e_{1,1} e_{1,j} \right]
 - \widehat{\delta}_{k,i}\delta_{j,1} \delta_{k,1} e_{1,1} e_{i,1},
\end{dmath}
\begin{dmath}
\label{h2}
e_{1,1}^2 e_{1,k} \cdot e_{i,1} e_{1,j}  = \left[\delta_{k,i}\delta_{j,1}\left(\delta_{k,1}+\tfrac{1}{2}\widehat{\delta}_{k,1}\right)+\tfrac{1}{2}\widehat{\delta}_{k,i} \delta_{k,1} \widehat{\delta}_{i,1} \delta_{i,j}\right]e_{1,1}
 +\tfrac{1}{2} \left(   \delta_{k,i}\widehat{\delta}_{k,1} \delta_{j,1} -  \widehat{\delta}_{k,i}\delta_{k,1}\widehat{\delta}_{i,1}\delta_{i,j}   \right)e_{1,1}^3
+ \left( \delta_{k,1} \left( \delta_{k,i} \widehat{\delta}_{j,1} -\widehat{\delta}_{k,i}  \delta_{i,1} \right) + \widehat{\delta}_{k,1}\widehat{\delta}_{j,1}\right)e_{1,1}^2e_{1,j},
\end{dmath}
\begin{dmath}
\label{h1}
e_{1,1}^2 e_{1,k} \cdot e_{1,i} e_{j,1}  = \delta_{k,1}\left(- \delta_{i,1} e_{1,1}^2 e_{j,1}+ \widehat{\delta}_{i,1}\delta_{i,j} \tfrac{1}{2}\left(e_{1,1}^3+e_{1,1}\right) \right);\, {(i,j)\neq(1,1)},
\end{dmath}
\begin{dmath}
\label{De3}
e_{1,1}^2 e_{1,k}\cdot e_{1,j} e_{1,1} e_{\ell ,1} = \delta_{k,1}\delta_{j,1}  e_{1,1} e_{\ell,1}; \quad {\ell\neq 1},
\end{dmath}
\begin{dmath}
\label{h3}
e_{1,1}^2 e_{1,k}\cdot e_{1,1}^2 e_{1,j}= \delta_{k,1} e_{1,1}e_{1,j},
\end{dmath}
\begin{dmath}
\label{De4}
 e_{1,1}^2 e_{1,j} \cdot e_{1,1}^4  = \delta_{j,1} e_{1,1}^3.
 \end{dmath}
 \end{dgroup*}
\end{proposition}
\begin{proposition}\label{47}
Let $i,j,k,\ell \in \Omega$ and $\ell\neq 1$. Then in $\mathfrak{A}$, we have
\begin{dgroup*}
\begin{dmath}
\label{ss1}
 e_{1,k} e_{1,1} e_{\ell,1} \cdot e_{i,j}= \delta_{j,\ell}\left[\tfrac{1}{2}\left\{ \left(\delta_{k,1}\widehat{\delta}_{j,1}\delta_{i,1} + \widehat{\delta}_{i,1}\delta_{i,k}\right) e_{1,1}^4 + \left(\widehat{\delta}_{i,1}\delta_{i,k}- \delta_{k,1}\widehat{\delta}_{j,1}\delta_{i,1} \right)e_{1,1}^2\right\}
 - \widehat{\delta}_{i,1} \delta_{k,1} e_{1,1} e_{i,1} - \widehat{\delta}_{k,1}e_{1,k} e_{i,1}\right],
\end{dmath}
\begin{dmath}
\label{ss2}
e_{1,k} e_{1,1} e_{\ell,1} \cdot e_{1,i} e_{j,1} = \delta_{i,\ell}\left[\delta_{k,1}\left(\delta_{j,1}\tfrac{1}{2}\left(e_{1,1}- e_{1,1}^3 \right) -\widehat{\delta}_{j,1}e_{1,1}^2 e_{j,1}\right)- \widehat{\delta}_{k,1}\left(\widehat{\delta}_{j,1} e_{1,k} e_{1,1} e_{j,1}+ \delta_{j,1} \left(e_{1,1}^2 e_{1,k}-e_{1,k}\right) \right)\right];\\\quad {(i,j)\neq (1,1)},
\end{dmath}
\begin{dmath}
\label{mm}
e_{1,i} e_{1,1} e_{j,1}\cdot e_{1,k} e_{1,1} e_{\ell,1}=\delta_{j,k}\left( e_{1,i} e_{\ell,1}- \widehat{\delta}_{i,1} \delta_{i,\ell} \tfrac{1}{2}\left( e_{1,1}^4 + e_{1,1}^2\right) \right);\quad {j\neq 1},
\end{dmath}
\begin{dmath}
\label{De1}
e_{1,k} e_{1,1} e_{\ell,1} \cdot e_{i,1} e_{1,j} = 0, \quad
  {e_{1,k} e_{1,1} e_{\ell,1}\cdot e_{1,1}^2 e_{1,j} = 0, \quad
   e_{1,k} e_{1,1} e_{\ell,1} \cdot e_{1,1}^4 = 0}.
    \end{dmath}
 \end{dgroup*}
\end{proposition}
\begin{proposition}\label{last}
Let $i,j \in \Omega$. Then in $\mathfrak{A}$, we have
\begin{dgroup*}
\begin{dmath}
\label{De6}
e_{1,1}^4 \cdot e_{i,j} = \delta_{i,1} \left(\delta_{j,1} e_{1,1} + \widehat{\delta}_{j,1} e_{1,1}^2 e_{1,j} \right)- \widehat{\delta}_{i,1} \delta_{j,1} e_{1,1}^2 e_{j,1},
\end{dmath}
\begin{dmath}
e_{1,1}^4 \cdot e_{1,i} e_{j,1} = \delta_{i,1} e_{1,1} e_{j,1} + \widehat{\delta}_{i,1} \widehat{\delta}_{j,1}\delta_{i,j} \tfrac{1}{2}\left(e_{1,1}^4 + e_{1,1}^2\right);\quad {(i,j)\neq (1,1)}, \label{k10}
\end{dmath}
\begin{dmath}
\label{De7}
e_{1,1}^4 \cdot e_{i,1} e_{1,j} = \delta_{i,1} e_{1,1} e_{1,j} + \widehat{\delta}_{i,1}\delta_{i,j} \tfrac{1}{2}\left(e_{1,1}^2 - e_{1,1}^4\right),
\end{dmath}
\begin{dmath}
\label{De8}
e_{1,1}^4 \cdot e_{1,1}^2 e_{1,j}  = e_{1,1}^2 e_{1,j},
\end{dmath}
\begin{dmath}
\label{De9}
e_{1,1}^4 \cdot e_{1,i} e_{1,1} e_{j,1} = \delta_{i,1} e_{1,1}^2 e_{j,1}; \quad {j\neq 1},
\end{dmath}
\begin{dmath}
\label{De5}
e_{1,1}^4 \cdot e_{1,1}^4  = e_{1,1}^4.
\end{dmath}
\end{dgroup*}
\end{proposition}


\section{The center of the universal enveloping algebra $\mathfrak{A}$}

Our next aim is to use the results of Section \ref{stc} to determine the center of $\mathfrak{A}$:
  \[
  Z(\mathfrak{A})= \{z \in \mathfrak{A}\,|\, zu = uz, \; \text{for all}\,\, u \in  \mathfrak{A}\}.
  \]

\begin{theorem}
The center $ Z(\mathfrak{A})$ of the (unital) universal enveloping algebra $\mathfrak{A}$ has dimension $5$ with basis:
\begin{align*}
&z_1 =\tfrac{(n-2)}{n} e_{1,1}^2- \tfrac{2}{n} \sum^n_{i=2} e_{1,i} e_{i,1} + e_{1,1}^4,\quad
 z_{2} = {\tiny{(2 -n)}} e_{1,1}^2 +\sum^n_{i=2} e_{1,i} e_{i,1} + \sum^n_{i=2} e_{i,1} e_{1,i},
\\
& z_{3} = -\tfrac{1}{2}e_{1,1} + \tfrac{1}{2} e_{1,1}^3 + \sum^n_{i=2} e_{1,i} e_{1,1} e_{i,1},\quad
 z_{4} = \sum^n_{i=1} e_{i,i},\quad  z_{5} = 1.
\end{align*}
\end{theorem}
\begin{proof}
To get the center of $\mathfrak{A}$, it is sufficient to determine the elements of $\mathfrak{A}$ which commute with $e_{i,j}$, for all $i,j \in \Omega$.  Let
\begin{dmath*}
x  =  \sum^n_{i, j = 1} \zeta^{(i,j)}_{1} e_{i,j} + \sum^n_{i = 1} \sum^n_{j = 2}\zeta^{(i,j)}_{2}e_{1,i} e_{1,1} e_{j,1} + \sum^n_{j = 1} \zeta^{(j)}_{3}e_{1,1}^2 e_{1,j}
 +  \sum^n_{i,j = 1} \zeta^{(i,j)}_{4}e_{i,1} e_{1,j} + \!\!\!\!\!\!\!\!\!\!\!\sum^n_{\tiny \begin{array}{c} i,j =1 \\ (i,j)\neq (1,1) \end{array}}\!\!\!\!\!\!\!\!\!\!\!\zeta^{(i,j)}_{5} e_{1,i} e_{j,1} + \zeta e_{1,1}^4,
\end{dmath*}
be any element of $Z(\mathfrak{A})$. Then
\begin{dmath}\label{e11}
 0 = x\, e_{1,1} - e_{1,1}\, x
=   \sum^n_{i, j = 1} \zeta^{(i,j)}_{1} \left(e_{i,j}e_{1,1} - e_{1,1} e_{i,j}\right) + \sum^n_{i = 1} \sum^n_{j = 2}\zeta^{(i,j)}_{2}\left( e_{1,i} e_{1,1} e_{j,1}  e_{1,1} - e_{1,1} e_{1,i} e_{1,1} e_{j,1} \right) + \sum^n_{j = 1} \zeta^{(j)}_{3} \left( e_{1,1}^2 e_{1,j}e_{1,1} - e_{1,1}^3 e_{1,j}\right)
+  \sum^n_{i,j = 1} \zeta^{(i,j)}_{4}\left(e_{i,1} e_{1,j}e_{1,1} - e_{1,1} e_{i,1} e_{1,j}\right)  + \!\!\!\!\!\!\!\!\!\!\!\sum^n_{\tiny \begin{array}{c} i,j =1 \\ (i,j)\neq (1,1) \end{array}}\!\!\!\!\!\!\!\!\!\!\!\zeta^{(i,j)}_{5} \left(e_{1,i} e_{j,1} e_{1,1} - e_{1,1} e_{1,i} e_{j,1}\right).
\end{dmath}
Proposition \ref{sf} implies that  $e_{i,1} e_{1,j} e_{1,1} = 0 = e_{1,i} e_{j,1} e_{1,1}$ for  $i\neq j \neq 1$,
 $e_{1,1} e_{i,1} e_{1,j}=0 = e_{1,1} e_{1,i} e_{j,1}$ for $1\neq i \neq j$,
 $e_{i,1} e_{1,i} e_{1,1}= \tfrac{1}{2} \left(e_{1,1}^3 - e_{1,1}\right) = e_{1,1} e_{i,1} e_{1,i}$ and
 $ e_{1,i} e_{i,1} e_{1,1}= \tfrac{1}{2} \left(e_{1,1}^3 + e_{1,1}\right)= e_{1,1} e_{1,i} e_{i,1}$ for $i\neq 1$. Using this in \eqref{e11} gives
\begin{dmath*}
 0 =  \sum^n_{i, j = 1} \zeta^{(i,j)}_{1} \left(e_{i,j}e_{1,1} - e_{1,1} e_{i,j}\right) -  \sum^n_{j = 2}\zeta^{(1,j)}_{2} e_{1,1}^3 e_{j,1}- \sum^n_{j = 2} \zeta^{(j)}_{3}  e_{1,1}^3 e_{1,j}
+  \sum^n_{i=2} \zeta^{(i,1)}_{4} e_{i,1} e_{1,1}^2-  \sum^n_{j=2} \zeta^{(1,j)}_{4} e_{1,1}^2 e_{1,j} + \sum^n_{i =2} \zeta^{(i,1)}_{5} e_{1,i} e_{1,1}^2 - \sum^n_{j =2} \zeta^{(1,j)}_{5} e_{1,1}^2 e_{j,1}.
\end{dmath*}
Using \eqref{l1}, \eqref{l2} of Proposition \ref{sf} and \eqref{g1}, \eqref{g2} of Proposition \ref{ss} implies
\begin{dmath*}
0 =   \sum^n_{j = 2} \zeta^{(1,j)}_{1} (e_{1,j} e_{1,1}-e_{1,1}e_{1,j})+\sum_{i=2}^n  \zeta^{(i,1)}_{1}(e_{i,1}e_{1,1} - e_{1,1} e_{i,1})+ \sum^n_{j = 2}\zeta^{(1,j)}_{2} e_{1,1} e_{j,1} - \sum^n_{j = 2} \zeta^{(j)}_{3} e_{1,1} e_{1,j}
+  \sum^n_{i=2} \zeta^{(i,1)}_{4} ( e_{1,1}^2 e_{i,1} + e_{i,1})- \sum^n_{j=2}\zeta^{(1,j)}_{4} e_{1,1}^2 e_{1,j} +\sum^n_{i=2} \zeta^{(i,1)}_{5}  (e_{1,1}^2 e_{1,i} - e_{1,i}) - \sum^n_{j=2} \zeta^{(1,j)}_{5} e_{1,1}^2 e_{j,1}.
\end{dmath*}
Comparing the coefficients on both sides, we get
\begin{align*} \zeta^{(1,j)}_{1}= \zeta^{(i,1)}_{1}= \zeta^{(1,j)}_{2} = \zeta^{(j)}_{3}= \zeta^{(i,1)}_{4}=  \zeta^{(1,i)}_{4}=\zeta_{5}^{(i,1)} = \zeta_{5}^{(1,j)} = 0,\end{align*}
for all  $i, j\in \Omega \setminus \{1\}$. Rewriting $x$ with these values for the coefficients,
we obtain
\begin{dmath*}
x  =  \zeta^{(1,1)}_{1} e_{1,1}+ \sum^n_{i, j = 2} \zeta^{(i,j)}_{1} e_{i,j} + \sum^n_{i,j = 2}\zeta^{(i,j)}_{2}e_{1,i} e_{1,1} e_{j,1} +  \zeta^{(1)}_{3}e_{1,1}^3
 + \zeta^{(1,1)}_{4}e_{1,1}^2 +  \sum^n_{i,j = 2} \zeta^{(i,j)}_{4}e_{i,1} e_{1,j} + \sum^n_{ i,j = 2}\zeta^{(i,j)}_{5} e_{1,i} e_{j,1} + \zeta e_{1,1}^4.
\end{dmath*}

Choose $ q\neq 1$ and observe that $e_{1, 1} e_{q,q} = 0 = e_{q,q} e_{1,1}$ by \eqref{l1} of Proposition \ref{sf}. Hence,
\begin{dmath*}
0 = x\, e_{q,q} - e_{q,q}\, x
=  \sum^n_{i, j = 2} \zeta^{(i,j)}_{1} \left(e_{i,j}e_{q,q} - e_{q,q} e_{i,j}\right) + \sum^n_{i,j = 2}\zeta^{(i,j)}_{2}\left(e_{1,i} e_{1,1} e_{j,1} e_{q,q} - e_{q,q} e_{1,i} e_{1,1} e_{j,1}\right)
 +  \sum^n_{i,j = 2} \zeta^{(i,j)}_{4}\left(e_{i,1} e_{1,j}e_{q,q} - e_{q,q}e_{i,1} e_{1,j} \right)+ \sum^n_{ i,j = 2}\zeta^{(i,j)}_{5} \left(e_{1,i} e_{j,1}e_{q,q} - e_{q,q}e_{1,i} e_{j,1}\right).
\end{dmath*}
Using Proposition \ref{sf}, \eqref{ss1} of Proposition \ref{47} and \eqref{g1} of Proposition \ref{ss} implies
\begin{dmath*}
0  =  \sum^n_{\tiny \begin{array}{c} i = 2 \\ i\neq q \end{array}} \zeta^{(i,q)}_{1}\left(e_{i,1}e_{1,q} - e_{1,q} e_{i,1}\right)+  \sum^n_{\tiny \begin{array}{c} j =2 \\ j\neq q \end{array}} \zeta^{(q,j)}_{1} \left(e_{1,j}e_{q,1} - e_{q,1} e_{1,j}\right)
- \!\!\! \sum^n_{\tiny \begin{array}{c}  i = 2\\ i\neq q\end{array}}\!\!\! \zeta^{(i,q)}_{2}   e_{1,i} e_{q,1} + \!\!\! \sum^n_{\tiny \begin{array}{c}  j = 2\\ j\neq q\end{array}}\!\!\! \zeta^{(q,j)}_{2}   e_{1,q} e_{j,1}
 + \!\!\! \sum^n_{\tiny \begin{array}{c}  i = 2\\ i\neq q\end{array}} \!\!\! \zeta^{(i,q)}_{4} \left(e_{1,q}e_{1,1} e_{i,1}+ e_{i,q}\right)-\!\!\! \sum^n_{\tiny \begin{array}{c}  j = 2\\ j\neq q\end{array}} \!\!\! \zeta^{(q,j)}_{4}\left( e_{1,j}e_{1,1} e_{q,1} +e_{q,j} \right)
+\!\!\! \sum^n_{\tiny \begin{array}{c}  i = 2\\ i\neq q\end{array}} \zeta^{(i,q)}_{5}  e_{1,i} e_{1,1} e_{q,1} - \!\!\! \sum^n_{\tiny \begin{array}{c}  j = 2\\ j\neq q\end{array}}\!\!\!  \zeta_5^{(q,j)} e_{1,q} e_{1,1}e_{j,1}.
\end{dmath*}
Comparing the coefficients on both sides gives
\begin{align*}
\zeta^{(i,q)}_{1} =  \zeta^{(q,j)}_{1} = \zeta^{(i,q)}_{2}=\zeta^{(q,j)}_{2} = \zeta^{(i,q)}_{4}=\zeta^{(q,j)}_{4}=\zeta^{(i,q)}_{5}=\zeta_5^{(q,j)} = 0,
\end{align*}
for all  $i,j, q\in \Omega\setminus\{1\}$ and $i\neq q\neq j$. Rewriting $x$ with these values for the coefficients,
we get
\begin{dmath*}
x  = \sum^n_{i = 1} \left(\zeta^{(i,i)}_{1} e_{i,i} + \zeta^{(i,i)}_{4}e_{i,1} e_{1,i} \right)+ \sum^n_{i= 2}\left(\zeta^{(i,i)}_{2}e_{1,i} e_{1,1} e_{i,1} + \zeta^{(i,i)}_{5} e_{1,i} e_{i,1}\right) +  \zeta^{(1)}_{3}e_{1,1}^3
  + \zeta e_{1,1}^4.
\end{dmath*}

We next choose $q,s\in \Omega\setminus\{1\}$ and $q\neq s$ and observe that $e_{1,1} e_{q,s} = 0 = e_{q,s} e_{1,1}$ by \eqref{l1} of Proposition \ref{sf}. Hence,
\begin{dmath*}
0  = x e_{qs}  - e_{q,s} x
 = \sum^n_{i = 1} \left(\zeta^{(i,i)}_{1} \left(e_{i,i}e_{q,s}- e_{q,s} e_{i,i}\right) +   \zeta^{(i,i)}_{4}\left(e_{i,1} e_{1,i}e_{q,s} - e_{q,s} e_{i,1} e_{1,i}\right) \right)
 + \sum^n_{i= 2}\left(\zeta^{(i,i)}_{2}\left(e_{1,i} e_{1,1} e_{i,1}e_{q,s} - e_{q,s}e_{1,i} e_{1,1} e_{i,1}\right)  + \zeta^{(i,i)}_{5} (e_{1,i} e_{i,1}e_{q,s} - e_{q,s} e_{1,i} e_{i,1})\right).
\end{dmath*}
Using Proposition \ref{sf}, \eqref{ss1} of Proposition \ref{47} and \eqref{g1} of Proposition \ref{ss} gives
\begin{dmath*}
0 =\left( \zeta^{(q,q)}_{1} - \zeta^{(s,s)}_{1}\right) e_{q,1} e_{1,s} +\left( \zeta^{(s,s)}_{1} - \zeta^{(q,q)}_{1}\right)e_{1,s} e_{q,1}+  \left(\zeta^{(q,q)}_{4}  -  \zeta^{(s,s)}_{4} \right) \left( e_{1,s} e_{1,1} e_{q,1} +e_{q,s}\right)
 -\left( \zeta^{(s,s)}_{2}  - \zeta^{(q,q)}_{2}\right)e_{1,s} e_{q,1} +\left( \zeta^{(s,s)}_{5} - \zeta^{(q,q)}_{5}\right) e_{1,s} e_{1,1} e_{q,1}.
\end{dmath*}
Comparing the coefficients on both sides gives
\begin{align*}
&\zeta^{(q,q)}_{1}= \zeta^{(s,s)}_{1},\quad \zeta^{(q,q)}_{4} =  \zeta^{(s,s)}_{4},\quad
  \zeta^{(s,s)}_{2} =  \zeta^{(q,q)}_{2},
\quad \zeta^{(s,s)}_{5} =  \zeta^{(q,q)}_{5},
\end{align*}
for all $ q, s \in \Omega\setminus\{1\}$ and $q \neq s$. Hence the values of $\zeta^{(q,q)}_{1}, \zeta^{(q,q)}_{4},  \zeta^{(q,q)}_{2}$ and $\zeta^{(q,q)}_{5} $ $( \text{for all}\, q\in \Omega\setminus\{1\})$ do not depend on the value of $q$. We remove the exponents of  $\zeta^{(q,q)}_{1}, \zeta^{(q,q)}_{4},  \zeta^{(q,q)}_{2}$ and $\zeta^{(q,q)}_{5} $ and rewrite $x$ and obtain
\begin{dmath*}
x = \zeta_1^{(1,1)}e_{1,1} +  \zeta_{1}\sum^n_{i = 2} e_{i,i} + \zeta^{(1,1)}_{4} e_{1,1}^2 + \sum^n_{i= 2}\left(\zeta_{4}\,e_{i,1} e_{1,i}+ \zeta_{2}\,e_{1,i} e_{1,1} e_{i,1} + \zeta_{5}\, e_{1,i} e_{i,1}\right) +  \zeta^{(1)}_{3}e_{1,1}^3
  + \zeta e_{1,1}^4.
\end{dmath*}

We observe that $x$ is invariant under the action of the anti-automorphism $\eta$. Therefore, if $x$ commutes with $e_{1,q}$ (resp. $e_{q,1}$)($q\ne 1$) then $x$ commutes with $e_{q,1}$(resp. $e_{1,q}$). 
Choose $q \neq \Omega\setminus\{1\}$ and note that $e_{q,1} e_{i,1} = 0 =e_{i,1} e_{q,1} $ ($i \neq 1$) by \eqref{l1} of proposition \ref{sf}. Hence,
\begin{dmath*}
0 = xe_{q,1} - e_{q,1} x
= \zeta_1^{(1,1)}(e_{1,1}e_{q,1}- e_{q,1} e_{1,1})+ \zeta_{1}\sum^n_{i = 2}  \left(e_{i,i}e_{q,1} - e_{q,1}e_{i,i}\right) + \zeta^{(1,1)}_{4}\left( e_{1,1}^2e_{q,1} - e_{q,1} e_{1,1}^2\right) + \zeta_{4}\sum^n_{i = 2} e_{i,1} e_{1,i}e_{q,1}- \zeta_{2}\sum^n_{i= 2} e_{q,1}e_{1,i} e_{1,1} e_{i,1}
 - \zeta_{5} \sum^n_{i= 2} e_{q,1} e_{1,i} e_{i,1}+  \zeta^{(1)}_{3}\left(e_{1,1}^3 e_{q,1} - e_{q,1} e_{1,1}^3 \right)  + \zeta \left(e_{1,1}^4 e_{q,1} - e_{q,1} e_{1,1}^4\right).
\end{dmath*}
Using \eqref{l1}-\eqref{l22} of Proposition \ref{sf}, \eqref{De2} of Propositions \ref{466}, \eqref{De6} of Proposition \ref{last} and Proposition \ref{ss} gives
\begin{align*}
& 0 = \zeta_1^{(1,1)}(e_{1,1}e_{q,1}- e_{q,1} e_{1,1})+ \zeta_{1}\sum^n_{i = 2}  \left(e_{q,1}e_{1,1} - e_{1,1}e_{q,1}\right) - \zeta^{(1,1)}_{4}e_{q,1} \\& \quad +\zeta_{4} \left(n e_{1,1}^2 e_{q,1}+e_{q,1}\right)+ \zeta_{2}\, e_{1,1} e_{q,1}
-\zeta_{5} \left( n e_{1,1}^2 e_{q,1} +(n-1) e_{q,1} \right)\\&\quad    +  \zeta^{(1)}_{3}(- e_{1,1} e_{q,1} - e_{q,1} e_{1,1})- \zeta ( 2 e_{1,1}^2 e_{q,1} + e_{q,1}).
\end{align*}
Combining the coefficients gives
\begin{dmath*}
0 =\left (\zeta^{(1,1)}_1 -  \zeta_{1} +\zeta_{2} -  \zeta^{(1)}_{3}\right) e_{1,1}e_{q,1}
+\left(-\zeta_1^{(1,1)}+ \zeta_{1}- \zeta^{(1)}_{3} \right)e_{q,1} e_{1,1}
+ \left(-\zeta^{(1,1)}_{4}+ \zeta_{4} - (n-1)\zeta_{5} -\zeta \right)e_{q,1}
+\left( n\zeta_{4}- n \zeta_{5}-2 \zeta \right) e_{1,1}^2 e_{q,1}.
\end{dmath*}
Comparing the coefficients on both sides gives
\begin{align*}
&\zeta_1^{(1,1)}-  \zeta_{1} + \zeta_{2} -  \zeta^{(1)}_{3} = 0,\qquad
-\zeta_1^{(1,1)}+\zeta_{1}- \zeta^{(1)}_{3} = 0,
\\
& -\zeta^{(1,1)}_{4}+ \zeta_{4} - (n-1)\zeta_{5} -\zeta = 0,\qquad
 n\zeta_{4}- n \zeta_{5}-2 \zeta = 0.
\end{align*}
These equations can be reduced to the system
\begin{align*}
&\zeta_1^{(1,1)}-  \zeta_{1} + \zeta^{(1)}_{3} = 0, \quad
 \zeta_{2} - 2 \zeta^{(1)}_{3}= 0,
\\
&\zeta_4^{(1,1)}+(n-2) \zeta_{4} - \left(\frac{n-2}{n}\right)\zeta = 0,\quad
 \zeta_{5} - \zeta_{4}+ \frac{2}{n} \zeta = 0.
\end{align*}
This is a linear system of four equations in eight variables. Hence, there are four free variables.
Setting, \begin{dmath*}(\zeta, \zeta_{4}, \zeta_{2} , \zeta_{1} )= (1,0,0,0),\,\, (0,1,0,0),\,\, (0,0,1,0), \,\, (0,0,0,1),\end{dmath*}
in the last system gives
\begin{dmath*}{\left( \zeta^{(1,1)}_{1}, \zeta^{(1)}_{3}, \zeta^{(1,1)}_{4},\zeta_{5} \right)  =  \left(0, 0, \tfrac{n-2}{n}, \tfrac{-2}{n}\right),\,\, \left(0, 0, 2-n,1 \right), \,\,\left( -\tfrac{1}{2}, \tfrac{1}{2}, 0, 0 \right),\,\, \left( 1, 0, 0, 0 \right)},\end{dmath*}
respectively. Using these solutions in $x$ gives $z_1$, $z_2$, $z_3$, $z_4$ respectively.
\end{proof}


\section{Explicit decomposition of the universal enveloping algebra}

\begin{theorem}
The universal enveloping algebra  $\mathfrak{A}$ of the anti-Jordan triple system $\mathfrak{J}$ can be decomposed as follows:
\[ \mathfrak{A} = F \oplus M_{n,n}(F) \oplus M_{n,n}(F) \oplus M_{n,n}(F) \oplus M_{n,n}(F),\]
where $M_{n,n}$ is the ordinary associative algebra of all $n \times n$ matrices.
\end{theorem}
\begin{proof}
We define the first two sets of $n \times n$ matrix units. For all $k \in \{0,1\}$ and $i,j= 2,\dots,n$, we set
\begin{align*}
&B^{(k)}_{1,1} = \tfrac{1}{4} \left(e_{1,1}^4 + e_{1,1}^2 +(-1)^k\left( e_{1,1}^3 + e_{1,1}\right)\right),\\&
 B^{(k)}_{1,i} = e_{1,1} e_{1,i} + (-1)^k e_{1,1}^2 e_{1,i},
\\
& B^{(k)}_{i,1} = \tfrac{1}{4} \left( e_{i,1} e_{1,1}+(-1)^k(e_{1,1}^2 e_{i,1}+ e_{i,1}) \right),
\\
& B^{(k)}_{i,i} = \tfrac{1}{2} \left(\tfrac{1}{2}\left( e_{1,1}^4 -e_{1,1}^2\right)+ e_{i,1} e_{1,i} +(-1)^k \left(e_{1,i} e_{1,1} e_{i,1} +   e_{i,i}\right) \right),
\\
&B^{(k)}_{i,j} = \tfrac{1}{2} \left( e_{i,1} e_{1,j} + (-1)^k (e_{1,j} e_{1,1} e_{i,1}+ e_{i,j})\right);\quad i \neq j.
\end{align*}
We wish to show that for each $k \in \{0,1\}$, the elements $B^{(k)}_{i,j}; i, j = 1,\dots,n $, satisfy the multiplication table for matrix units and the product of any $B^{(k)}_{i,j}$
by any $B^{(s)}_{t,\ell}$ is $0$ for $k\neq s$. We note first that if $i,j\neq 1$ then  $B^{(k)}_{1,i} B^{(s)}_{1,j} = 0$, since $ e_{1,1} e_{1,i} e_{1,1} = 0$ by \eqref{l22} of Proposition \ref{sf}.
Let $i\neq 1$. Then
\begin{dmath*}
B^{(k)}_{1,1} B^{(s)}_{1,i}
 = \tfrac{1}{4}\left[ (e_{1,1}^4 + e_{1,1}^2) e_{1,1} e_{1,i}+ (-1)^s (e_{1,1}^4 + e_{1,1}^2)e_{1,1}^2e_{1,i}+(-1)^k\left(e_{1,1}^3 + e_{1,1}\right)e_{1,1}e_{1,i} +(-1)^{k+s}\left(e_{1,1}^3 + e_{1,1}\right)e_{1,1}^2e_{1,i}\right].
\end{dmath*}
Using \eqref{De7}, \eqref{De8}, \eqref{De6} of Proposition \ref{last} and \eqref{De2} of Proposition \ref{466} implies
\begin{dmath*}
 B^{(k)}_{1,1} B^{(s)}_{1,i}  = \tfrac{1}{4}\left[ 2\left( 1+ (-1)^{k+s}\right) e_{1,1} e_{1,i}+ 2\left((-1)^s + (-1)^{k}\right) e_{1,1}^2 e_{1,i} \right]{= \delta_{k,s} B^{(k)}_{1,i}}.
\end{dmath*}
Also,
\begin{dmath*}
 B_{1,1}^{(k)}  B_{i,1}^{(s)}  =
 \tfrac{1}{16} \left[ \left(e_{1,1}^4 + e_{1,1}^2\right) e_{i,1} e_{1,1} + (-1)^s \left( e_{1,1}^4 + e_{1,1}^2 \right) \left( e_{1,1}^2 e_{i,1}+ e_{i,1}\right) +(-1)^k \left(e_{1,1}^3 + e_{1,1}\right) e_{i,1} e_{1,1} + (-1)^{k+s}\left(e_{1,1}^3 + e_{1,1}\right) \left( e_{1,1}^2 e_{i,1} + e_{i,1}\right) \right].
\end{dmath*}
 Using \eqref{De6}, \eqref{k10}, \eqref{De9} of Proposition \ref{last} and \eqref{De2} of Proposition \ref{466} and observing that $e_{1,1} e_{i,1} e_{1,1} =0$ by \eqref{l2} of Proposition \ref{sf} imply
\begin{dmath*}
B_{1,1}^{(k)}  B_{i,1}^{(s)}  = \tfrac{1}{16} \left[ (-1)^s\left( e_{1,1}^2 e_{i,1} - e_{1,1}^2 e_{i,1} - e_{1,1}^2 e_{i,1}+ e_{1,1}^2 e_{i,1} \right)  +(-1)^{k+s}\left( e_{1,1}e_{i,1}-e_{1,1} e_{i,1} - e_{1,1} e_{i,1}+ e_{1,1} e_{i,1}\right)\right]
{= 0}.
\end{dmath*}
Next let $i,j \neq 1$. Then
\begin{dmath*}
B^{(k)}_{1,i} B^{(s)}_{j,1}  =
 \tfrac{1}{4} \left[ e_{1,1} e_{1,i} e_{j,1} e_{1,1} + (-1)^s e_{1,1} e_{1,i} \left(e_{1,1}^2 e_{j,1} + e_{j,1}\right)+(-1)^k e_{1,1}^2 e_{1,i} e_{j,1} e_{1,1} + (-1)^{k+s} e_{1,1}^2e_{1,i}\left(e_{1,1}^2 e_{j,1} + e_{j,1}\right) \right].
\end{dmath*}
Using \eqref{g4} of Proposition \ref{st}, \eqref{l3} of Proposition \ref{sf}, \eqref{De2} and \eqref{h2} of Proposition \ref{466} gives
\begin{dmath*}
B^{(k)}_{1,i} B^{(s)}_{j,1}  = \tfrac{1}{4} \delta_{j,i}\left[ \tfrac{1}{2} \left( e_{1,1}^4 + e_{1,1}^2 \right)+ \tfrac{1}{2}(-1)^s (e_{1,1}^3 + e_{1,1})+\tfrac{1}{2}(-1)^k \left(e_{1,1} + e_{1,1}^3\right) + \tfrac{1}{2}(-1)^{k+s} \left(e_{1,1}^4 + e_{1,1}^2\right)\right]
 {=\delta_{j,i} \delta_{k,s} B^{(k)}_{1,1}}.
\end{dmath*}
Also,
\begin{dmath*}
B^{(k)}_{i,1} B^{(s)}_{1,j} = \tfrac{1}{4}\left[ e_{i,1} e_{1,1}^2 e_{1,j} + (-1)^s e_{i,1} e_{1,1}^3 e_{1,j} + (-1)^k\left( e_{1,1}^2 e_{i,1} e_{1,1} e_{1,j} + e_{i,1}e_{1,1} e_{1,j}\right) +(-1)^{k+s}\left( e_{1,1}^2 e_{i,1} e_{1,1}^2 e_{1,j} + e_{i,1} e_{1,1}^2 e_{1,j}\right)\right]
= \tfrac{1}{4}\left[\left(1+(-1)^{k+s}\right) e_{i,1} e_{1,1}^2 e_{1,j} + \left((-1)^s + (-1)^k \right) e_{i,1}e_{1,1} e_{1,j}\right],
\end{dmath*}
since  $e_{1,1}^3 e_{1,j} = e_{1,1} e_{1,j}$ and $e_{1,1} e_{i,1} e_{1,1} = 0$. Using \eqref{g4} of Proposition \ref{st} and \eqref{l2} of proposition \ref{sf} implies
\begin{dmath*}
B^{(k)}_{i,1} B^{(s)}_{1,j} = \tfrac{1}{4}\left[ (1+(-1)^{k+s})\left(\delta_{j,i} \tfrac{1}{2} (e_{1,1}^4 - e_{1,1}^2) + e_{i,1} e_{1,j}\right)+((-1)^s +(-1)^k) \left( e_{1,j} e_{1,1} e_{i,1} + e_{i,j}\right) \right]{= \delta_{k,s} B^{(k)}_{i,j}}.
\end{dmath*}
We have shown that
\begin{align*}
&    B^{(k)}_{1,i} B^{(s)}_{1,j} = 0, \quad B^{(k)}_{1,1} B^{(s)}_{1,i}= \delta_{k,s} B^{(k)}_{1,i},\quad B_{1,1} ^{(k)}B^{(s)}_{i,1}=0,\\&
B^{(k)}_{1,i}B^{(s)}_{j,1}=\delta_{k,s} \delta_{j,i} B_{1,1}^{(k)}, \quad
 B^{(k)}_{i,1}B^{(s)}_{1,j}= \delta_{k,s} B^{(k)}_{i,j},
\end{align*}
for all $i,j\neq 1$. By applying the anti-automorphism $\eta$ to both sides of the first three products and observing that $ {B^{(k)}_{1,i}} =  4 \eta({B^{(k)}_{i,1}}) $, we obtain
\begin{align*}
& B^{(k)}_{j,1} B^{(s)}_{i,1}  = 0, \quad  B^{(s)}_{i,1} B^{(k)}_{1,1} = \delta_{k,s} B^{(k)}_{i,1}, \quad B^{(s)}_{1,i} B_{1,1} ^{(k)}=0.
\end{align*}
Now we use the above products to get all the others. For $k,s \in \{0,1\}$ and $i\neq 1$,
we have $ B^{(k)}_{1,i}B^{(k)}_{i,1}= B^{(k)}_{1,1} $, hence $ B_{1,1}^{(s)} B^{(k)}_{1,i} B^{(k)}_{i,1} = B^{(s)}_{1,1}B^{(k)}_{1,1}$.
Thus $B^{(s)}_{1,1} B^{(k)}_{1,1}= \delta_{k,s} B^{(k)}_{1,1}$.
We now have  $B^{(k)}_{i,q} =  B^{(k)}_{i,1} B^{(k)}_{1,q}  $ (for all $i,q$). Hence, $B^{(k)}_{i,q}B^{(s)}_{\ell,t} =  B^{(k)}_{i,1} B^{(k)}_{1,q}B^{(s)}_{\ell,1} B^{(s)}_{1,t} = \delta_{k,s} \delta_{q,\ell} B^{(k)}_{i,1} B^{(k)}_{1,1} B^{(k)}_{1,t}= \delta_{k,s} \delta_{q,\ell} B^{(k)}_{i,1} B^{(k)}_{1,t}=\delta_{k,s} \delta_{q,\ell} B^{(k)}_{i,t}$ (for all $i, q, \ell, t$).
Summarizing
\begin{equation}\label{BBM} B^{(s)}_{i,j} B^{(s)}_{t,\ell} = \delta_{j,t} B^{(s)}_{i,\ell},\quad   B^{(s)}_{i,j} B^{(k)}_{t,\ell} = 0,  \end{equation}
for all  $s,k \in \{0,1\}, s\neq k$ and $i,j,t,\ell = 1,\dots n$.

We define next the two other sets of $n \times n$ matrix units.
For $k\in \{0,1\}$ and $i,j =2,\dots,n$, we set
\begin{align*}
&D^{(k)}_{1,1} = \tfrac{1}{4} \left(e_{1,1}^4 - e_{1,1}^2 +(-1)^k \mathrm{I} \left(e_{1,1} - e_{1,1}^3\right)\right),
\\
& D^{(k)}_{1,i} = -\tfrac{1}{2} \left( e_{1,1} e_{i,1} +(-1)^k \mathrm{I}\, e_{1,1}^2 e_{i,1}\right),
\\
& D^{(k)}_{i,1} = -\tfrac{1}{2}  \left( e_{1,i} e_{1,1}+(-1)^k \mathrm{I} \left( e_{1,1}^2 e_{1,i}-e_{1,i}\right) \right),
\\
& D^{(k)}_{i,i} = \tfrac{1}{2} \left( \tfrac{1}{2} \left(e_{1,1}^4 + e_{1,1}^2\right) -  e_{1,i} e_{i,1}-(-1)^k \mathrm{I}\,  e_{1,i} e_{1,1} e_{i,1} \right),
\\
& D^{(k)}_{i,j} = - \tfrac{1}{2} \left( e_{1,i} e_{j,1} + (-1)^k \mathrm{I}\, e_{1,i} e_{1,1} e_{j,1} \right);\quad i\neq j,
\end{align*}
where  $\mathrm{I}=\sqrt{-1}$. We wish to show that for each $k \in \{0,1\}$, the elements $D^{(k)}_{i,j}; i, j = 1,\dots,n $, satisfy the multiplication table for matrix units and the product of any $D^{(k)}_{i,j}$
by any $D^{(s)}_{t,\ell}$ is $0$ for $k\neq s$. We note first that if $i,j \neq 1$ then $D^{(k)}_{1,i}D^{(s)}_{1,j} = 0$, since $e_{1,1} e_{i,1} e_{1,1} = 0$. Let $i\neq 1$. Then
\begin{dmath*}
D^{(k)}_{1,1} D^{(s)}_{1,i}= - \tfrac{1}{8}\left[(e_{1,1}^4 - e_{1,1}^2) e_{1,1} e_{i,1}+ (-1)^s \mathrm{I} \left(e_{1,1}^4 - e_{1,1}^2\right)e_{1,1}^2e_{i,1}
 +(-1)^{k} \mathrm{I}\left(e_{1, 1} - e_{1,1}^3\right)e_{1,1} e_{i,1}   - (-1)^{k+ s} \left(e_{1,1}- e_{1,1}^3\right)e_{1,1}^2 e_{i,1}\right].
\end{dmath*}
Using \eqref{k10}, \eqref{De9}, \eqref{De6} of Proposition \ref{last} and \eqref{De2} of Proposition \ref{466} gives
\begin{dmath*}
D^{(k)}_{1,1} D^{(s)}_{1,i}  =  - \tfrac{1}{8}\left[2e_{1,1} e_{i,1}+ 2 (-1)^s \mathrm{I}\, e_{1,1}^2e_{i,1} + 2(-1)^{k} \mathrm{I}\, e_{1, 1}^2e_{i,1}  + 2(-1)^{k+ s} e_{1,1} e_{i,1}\right]
 = - \tfrac{1}{4}\left[\left(1+ (-1)^{k+ s}\right) e_{1,1} e_{i,1}+ \left((-1)^s+ (-1)^{k}\right) \mathrm{I}\, e_{1,1}^2e_{i,1}  \right]
 {=\delta_{k,s} D^{(k)}_{1,i}}.
\end{dmath*}
Also,
\begin{dmath*}
D^{(k)}_{1,1} D^{(s)}_{i,1}  = -\tfrac{1}{8}\left[ e_{1,1}^4 e_{1,i} e_{1,1} +(-1)^s \mathrm{I}\, e_{1,1}^4\left( e_{1,1}^2 e_{1,i} - e_{1,i}\right)- e_{1,1}^2 e_{1,i} e_{1,1} -(-1)^s\mathrm{I}\, e_{1,1}^2 (e_{1,1}^2 e_{1,i} - e_{1,i}) +(-1)^{k}\mathrm{I} \left(e_{1,1} - e_{1,1}^3\right) e_{1,i} e_{1,1} - (-1)^{k+s}\left(e_{1,1} -e_{1,1}^3\right)\left(e_{1,1}^2 e_{1,i}-e_{1,i}\right)\right].
\end{dmath*}
Using \eqref{De6}, \eqref{De7}, \eqref{De8} of Proposition \ref{last} and  \eqref{De2} of Proposition \ref{466} and observing that $ e_{1,1} e_{1,i} e_{1,1} = 0$ imply
\begin{dmath*}
D^{(k)}_{1,1} D^{(s)}_{i,1} = -\tfrac{1}{8}\left[ (-1)^s \mathrm{I} \left( e_{1,1}^2 e_{1,i} -e_{1,1}^2 e_{1,i}\right) -(-1)^s \mathrm{I} \left(e_{1,1}^2 e_{1,i} - e_{1,1}^2e_{1,i}\right)- (-1)^{k+s}\left(e_{1,1}e_{1,i} -  e_{1,1} e_{1,i} - e_{1,1} e_{1,i} + e_{1,1} e_{1,i}\right)\right]
 {= 0}.
\end{dmath*}
Next let $i,j\neq 1$. Then
\begin{dmath*}
D^{(k)}_{1,i} D^{(s)}_{j,1} = \tfrac{1}{4} \left[ e_{1,1} e_{i,1} e_{1,j} e_{1,1} +(-1)^s \mathrm{I} \left( e_{1,1} e_{i,1}e_{1,1}^2 e_{1,j}- e_{1,1} e_{i,1} e_{1,j}\right) +(-1)^k \mathrm{I}\, e_{1,1}^2 e_{i,1} e_{1,j} e_{1,1}
-(-1)^{k+s} e_{1,1}^2 e_{i,1} \left(e_{1,1}^2 e_{1,j} - e_{1,j}\right)\right].
\end{dmath*}
Using \eqref{k1} of Proposition \ref{LL}, \eqref{l2} of Proposition \ref{sf}, \eqref{ss1} and \eqref{ss2} of Proposition \ref{47} gives
\begin{dmath*}
D^{(k)}_{1,i} D^{(s)}_{j,1} =  \tfrac{1}{4} \left[ \delta_{i,j}\tfrac{1}{2}\left(e_{1,1}^4 - e_{1,1}^2\right) -\delta_{i,j} (-1)^s \mathrm{I}\tfrac{1}{2} \left(e_{1,1}^3 - e_{1,1}\right) +  \delta_{i,j}(-1)^k\tfrac{1}{2} \mathrm{I} (e_{1,1} - e_{1,1}^3)
+  \delta_{i,j} (-1)^{k+s} \tfrac{1}{2}  (e_{1,1}^4 - e_{1,1}^2) \right]
 = \tfrac{1}{4}\delta_{i,j} \left[ \tfrac{1}{2} \left(1+ (-1)^{k+s}\right)\left(e_{1,1}^4 - e_{1,1}^2\right)+  \tfrac{1}{2}\left((-1)^s + (-1)^k\right) \mathrm{I} (e_{1,1} - e_{1,1}^3)
 \right]
= \delta_{s,k} \delta_{i,j} D_{1,1}^{(k)}.
\end{dmath*}
Also,
\begin{dmath*}
 D^{(k)}_{i,1} D^{(s)}_{1,\ell} = \tfrac{1}{4}\left[ e_{1,i} e_{1,1}^2 e_{\ell,1} + (-1)^s \mathrm{I}\, e_{1,i} e_{1,1}^3 e_{\ell,1} + (-1)^{k} \mathrm{I} \left(e_{1,1}^2 e_{1,i}e_{1,1} e_{\ell,1} - e_{1,i}e_{1,1}e_{\ell,1}\right) -(-1)^{k+s}\left(e_{1,1}^2 e_{1,i} e_{1,1}^2 e_{\ell,1} - e_{1,i} e_{1,1}^2 e_{\ell,1}\right)\right].
\end{dmath*}
Using \eqref{g1} of Proposition \ref{ss} and \eqref{k2} of Proposition \ref{LL} implies
\begin{dmath*}
 D^{(k)}_{i,1} D^{(s)}_{1,\ell} = \tfrac{1}{4}\left[(1+(-1)^{k+s}) \left( \delta_{i,\ell}\tfrac{1}{2} (e_{1,1}^4 + e_{1,1}^2) - e_{1,i} e_{\ell,1}\right) -\left((-1)^s+(-1)^{k}
 \right) \mathrm{I}\, e_{1,i} e_{1,1} e_{\ell,1} \right] {= \delta_{k,s} D^{(k)}_{i,\ell}}.
\end{dmath*}
The other products can be obtained by using the argument at the end of the proof of the first two sets of $n \times n$ matrix units. Summarizing
\begin{equation}\label{BBM1} D^{(s)}_{i,j} D^{(s)}_{k,\ell} = \delta_{j,k} D^{(s)}_{i,\ell},\quad   D^{(s)}_{i,j} D^{(t)}_{k,\ell} = 0,\end{equation}
for all $s,t \in \{0,1\}$, $s\neq t$ and  $i,j,k,\ell= 1,\dots,n$.

We wish to prove now that the product of any $D^{(k)}_{i,j}$ by any  $B^{(s)}_{m,n}$ is $0$. Clearly $D_{1,i}^{(k)} B^{(s)}_{\ell,1}=0 $ and $ D^{(k)}_{1,i} B^{(s)}_{1,\ell}= 0$ ($i,\ell\neq 1$), since $e_{i,1} e_{\ell,1} = 0$ and $e_{1,1} e_{i,1} e_{1,1} = 0$. Let $\ell \neq 1$. Then
\begin{dmath*}
D_{1,1}^{(k)} B^{(s)}_{1,\ell} = \tfrac{1}{4}\left[\left(e_{1,1}^4 - e_{1,1}^2\right) e_{1,1}e_{1,\ell}+(-1)^s (e_{1,1}^4 - e_{1,1}^2) e_{1,1}^2 e_{1,\ell} +(-1)^{k}\mathrm{I}\left(e_{1,1}-e_{1,1}^3\right)e_{1,1} e_{1,\ell} +(-1)^{k+s} \mathrm{I} \left(e_{1,1}- e_{1,1}^3 \right)e_{1,1}^2e_{1,\ell} \right]
 =  \tfrac{1}{4}\left[\left( e_{1,1} e_{1,\ell} -e_{1,1}e_{1,\ell} \right)+(-1)^s \left(e_{1,1}^2e_{1,\ell} -e_{1,1}^2e_{1,\ell}\right) +(-1)^{k}\mathrm{I}\left(e_{1,1}^2-e_{1,1}^2\right) e_{1,\ell} +(-1)^{k+s} \mathrm{I} \left(e_{1,1}- e_{1,1}\right)e_{1,\ell} \right] {= 0},
\end{dmath*}
using \eqref{De7}, \eqref{De8}, \eqref{De6} of Proposition \ref{last} and \eqref{De2} of Proposition \ref{466}.
Also,
\begin{dmath*}
D^{(k)}_{1,1} B^{(s)}_{\ell,1} = \tfrac{1}{16}\left[ (e_{1,1}^4- e_{1,1}^2)e_{\ell,1} e_{1,1} + (-1)^s \left(e_{1,1}^4 - e_{1,1}^2\right)\left( e_{1,1}^2 e_{\ell,1} + e_{\ell,1}\right) +(-1)^k \mathrm{I} \left(e_{1,1} - e_{1,1}^3\right)e_{\ell,1} e_{1,1}+(-1)^{k+s}\mathrm{I} \left(e_{1,1} - e_{1,1}^3\right)\left(e_{1,1}^2 e_{\ell,1}+e_{\ell,1}\right)\right]
=\tfrac{1}{16}\left[(-1)^s \left(e_{1,1}^2 e_{\ell,1}- e_{1,1}^2 e_{\ell,1} - \left( -e_{1,1}^2 e_{\ell,1} + e_{1,1}^2 e_{\ell,1}\right) \right) +(-1)^{k+s}\mathrm{I} \left(-e_{1,1} e_{\ell,1}+e_{1,1}e_{\ell,1} -  e_{1,1} e_{\ell,1}+ e_{1,1} e_{\ell,1}\right)\right] {= 0},
\end{dmath*}
using \eqref{De9}, \eqref{De6}, \eqref{k10} of Proposition \ref{last} and  \eqref{De2} of Proposition \ref{466}.
We have shown that
\begin{align}\label{DBB}
D_{1,i}^{(k)} B^{(s)}_{\ell,1}=0 ,\quad  D^{(k)}_{1,i} B^{(s)}_{1,\ell}= 0,\quad D_{1,1}^{(k)} B^{(s)}_{1,\ell}=0,\quad D^{(k)}_{1,1} B^{(s)}_{\ell,1}=0,\end{align}
for all $i,\ell \neq 1$ and $k,s \in \{0,1\}$.  Let $\ell\neq 1$, then  $D^{(k)}_{1,i} B^{(s)}_{1,1} = D^{(k)}_{1,i} B^{(s)}_{1,\ell}B^{(s)}_{\ell,1} = 0$ (for all $i$), using \eqref{BBM} and \eqref{DBB}. Combining this result with the first and the last equations of \eqref{DBB} gives $D^{(k)}_{1,i} B^{(s)}_{j,1} = 0$ (for all $i,j$).  By \eqref{BBM} and \eqref{BBM1},  $D^{(k)}_{i,j} B^{(s)}_{t,\ell} = D^{(k)}_{i,1} D^{(k)}_{1,j} B^{(s)}_{t,1} B^{(s)}_{1,\ell}$ (for all $i,j,t,\ell$). Hence, $D^{(k)}_{i,j} B^{(s)}_{t,\ell} =0$ (for all $i,j,t,\ell$).
By using the anti-automorphism $\eta$, we can show that $ B^{(s)}_{t,\ell} D^{(k)}_{i,j} = 0$ (for all $i,j,t,\ell$). Summarizing
\begin{equation}\label{BBM2} D^{(k)}_{i,j}  B^{(s)}_{t,\ell} = 0 =  B^{(s)}_{t,\ell}D^{(k)}_{i,j} \quad \text{for all} \,\, i,j,t,\ell= 1,\dots,n, \,\, s,k \in \{0,1\}.\end{equation}

Finally, we define the set of $1 \times 1$ matrix unit.
We set,
\begin{align}\label{AAu}
A_{1,1} = \sum^n_{i=2}e_{1,i} e_{i,1} - \sum^n_{i=2} e_{i,1} e_{1,i} - ne_{1,1}^4 +1.
\end{align}
We wish to show that $A_{1,1}^2 = A_{1,1}$ and the products of $A_{1,1}$ by any $B^{(k)}_{i,j}$ and  $D^{(k)}_{i,j}$ are $0$. We observe that
\begin{dmath}\label{kj}
\sum^1_{k=0}\left(B^{(k)}_{1,1}+ D^{(k)}_{1,1} \right)+ \sum^n_{i=2}\sum^1_{k=0}B^{(k)}_{i,i}+ \sum^n_{i=2}\sum^1_{k=0}D^{(k)}_{i,i}
= e^4_{1,1}+ \sum^n_{i=2} \left(\tfrac{1}{2} (e_{1,1}^4 - e_{1,1}^2) + e_{i,1}e_{1,i}\right)+\sum^{n}_{i=2} \left(\tfrac{1}{2}( e_{1,1}^4 + e_{1,1}^2) - e_{1,i} e_{i,1}\right)
 = e^4_{1,1}+ \tfrac{1}{2}(n-1) (e_{1,1}^4 - e_{1,1}^2) + \sum^{n}_{i=2} e_{i,1} e_{1,i}+\tfrac{1}{2} (n-1)(e_{1,1}^4 + e_{1,1}^2) - \sum^n_{i=2}e_{1,i} e_{i,1}
 = ne_{1,1}^4 + \sum^{n}_{i=2} e_{i,1} e_{1,i}- \sum^n_{i=2}e_{1,i} e_{i,1}.
\end{dmath}
Using \eqref{kj} in \eqref{AAu} gives
\begin{equation}\label{Aq} A_{1,1} = 1 - \left(\sum^1_{k=0}\left(B^{(k)}_{1,1}+ D^{(k)}_{1,1} \right)+ \sum^n_{i=2}\sum^1_{k=0}B^{(k)}_{i,i}+ \sum^n_{i=2}\sum^1_{k=0}D^{(k)}_{i,i}\right).\end{equation}
Multiply \eqref{Aq} by  $B^{(k)}_{\ell, m}$ from the right and use the relations of \eqref{BBM}, \eqref{BBM1} and \eqref{BBM2} (of the present proof), we obtain
\[
A_{1,1} B^{(k)}_{\ell, m} = B^{(k)}_{\ell, m}- B^{(k)}_{\ell, m}=0.
\]
Similarly, we can show that $B^{(k)}_{\ell, m}A_{1,1} =0$ and $ A_{1,1} D^{(k)}_{\ell, m}=0= D^{(k)}_{\ell, m}A_{1,1}$.
To show $ A^2_{1,1} = A_{1,1}$, we multiply \eqref{Aq}  by $A_{1,1}$ and use the last discussion.

Now let $\Phi^{(k)}_{n}$ (resp. $\Psi^{(k)}_{n}$ and $\tau_1$) denote the subspace of $\mathfrak{A}$ generated by the $B^{(k)}_{i,j}$ (resp. $D^{(k)}_{i,j}$ and $A_{1,1}$), $k\in \{0,1\}$.
Our discussion shows that $\Phi^{(k)}_{n}$ (resp. $\Psi^{(k)}_{n}$ and $\tau_1$) is a subalgebra of $\mathfrak{A}$ and isomorphic to $M_{n,n}$ (resp. $M_{n,n}$ and $M_{1,1}$),   $\Phi^{(k)}_{n} \Phi^{(s)}_{n} = 0 = \Phi^{(s)}_{n}\Phi^{(k)}_{n} $, $\Psi^{(k)}_{n} \Psi^{(s)}_{n} = 0 = \Psi^{(s)}_{n}\Psi^{(k)}_{n} $ $(k \neq s)$, $\Phi^{(k)}_{n} \Psi^{(s)}_{n} = 0 = \Psi^{(s)}_{n}\Phi^{(k)}_{n} $, $\Phi^{(s)}_{n} \tau_1 = 0 = \tau_1 \Phi^{(s)}_{n}$ and $ \Psi^{(s)}_{n} \tau_1 = 0 = \tau_1 \Psi^{(s)}_{n}$.
By \eqref{Aq} and the definitions of $B^{(k)}_{i,j}$ and $D^{(k)}_{i,j}$, we have
\begin{dmath}\label{re}
 1 = A_{1,1} + \sum^{n}_{i=1}B^{(0)}_{i,i}+\sum^{n}_{i=1}B^{(1)}_{i,i}+ \sum^{n}_{i=1} D^{(0)}_{i,i}+\sum^{n}_{i=1} D^{(1)}_{i,i},
 \\
{e_{1,1} =  B^{(0)}_{1,1}- B^{(1)}_{1,1} - \mathrm{I}\, D^{(0)}_{1,1}+ \mathrm{I}\,  D^{(1)}_{1,1}},
\\
 {e_{i,j} = B^{(0)}_{i,j}- B^{(1)}_{i,j} - \mathrm{I}\, D^{(0)}_{j,i}+ \mathrm{I}\,  D^{(1)}_{j,i};\quad {i,j \neq 1,  i\neq j}},
 \\
  {e_{i,i} = B^{(0)}_{i,i}- B^{(1)}_{i,i} - \mathrm{I}\, D^{(0)}_{i,i}+ \mathrm{I}\,  D^{(1)}_{i,i}; \quad {i\neq 1}},
  \\
 {e_{1,i} =  \tfrac{1}{2} B^{(0)}_{1,i}- \tfrac{1}{2} B^{(1)}_{1,i} - \mathrm{I}\, D^{(0)}_{i,1}+ \mathrm{I}\,  D^{(1)}_{i,1}; \quad i\neq 1},
 \\
{e_{i,1} = 2 B^{(0)}_{i,1}-2 B^{(1)}_{i,1} - \mathrm{I}\, D^{(0)}_{1,i}+ \mathrm{I}\,  D^{(1)}_{1,i}; \quad {i\neq 1}}.
\end{dmath}
Thus all the $1, e_{i,j} \in  \tau_{1} \oplus \Phi^{(0)}_{n} \oplus \Phi^{(1)}_{n}\oplus \Psi^{(0)}_{n}\oplus \Psi^{(1)}_{n} $. Hence $ \mathfrak{A} = \tau_{1} \oplus \Phi^{(0)}_{n} \oplus \Phi^{(1)}_{n}\oplus \Psi^{(0)}_{n}\oplus \Psi^{(1)}_{n}$.
\end{proof}
\begin{remark}
The equations \eqref{re} (of the last proof) describe all inequivalent irreducible representations of the anti-Jordan triple system $\mathfrak{J}$.
\end{remark}
\begin{corollary}
The universal enveloping algebra of the simple anti-Jordan triple system of all $n \times n$ matrices over an algebraically closed field is semisimple.
\end{corollary}

The next example shows that the universal enveloping algebra is not necessary to be finite-dimensional
\begin{example}
Consider the  2-dimensional anti-Jordan triple system $S$ with basis $\mathcal{B}= \{ a= e_{1,2}, b= e_{2,1}\}$ of matrix units and triple product given by $\langle a,b,c   \rangle= abc- cba$. It is easy to check that the multiplication table of $S$ is zero.  The universal enveloping algebra is associative algebra with relations: $b^2 a = a b^2$ and $ba^2 = a^2 b$, which is the down-up algebra $A(0,1,0)$ (see \cite{roby}).
\end{example}

To conclude the paper, we formulate the following conjecture.

\begin{conjecture}
If the universal enveloping algebra of a simple finite-dimensional anti-Jordan triple system is finite-dimensional, then it is semisimple.
\end{conjecture}


\section*{Acknowledgements}
This work forms part of the author's doctoral thesis. The author would like to thank her supervisor Prof. Murray Bremner for very helpful discussions and comments.
The author was supported by a Teacher Scholar Doctoral Fellowship from the
University of Saskatchewan.


\end{document}